\documentclass{article}
\usepackage{graphicx}
\usepackage{epstopdf}
\usepackage{amsmath}
\usepackage{amssymb}
\usepackage{bm}
\usepackage{bbm}
\usepackage{enumerate}
\usepackage[T1]{fontenc}
\usepackage[latin1]{inputenc}
\usepackage{hyperref}
\usepackage{color}
\usepackage{curves}
\usepackage{tikz}
\usepackage[hang]{caption}
\usepackage{ntheorem}
\usepackage{float}
\usepackage{mathrsfs}

\title{Water transport on graphs}
\author{Olle Häggström \thanks{Research supported by grants
 from the Swedish Research Council and from the Knut and Alice
 Wallenberg Foundation}
\qquad Timo Hirscher
     \thanks{Research supported by grants from the Swedish Research Council and the Royal Swedish Academy of Sciences}
\\\normalsize Chalmers University of Technology}

\theoremstyle{break}
\newtheorem{theorem}{Theorem}[section]

\newtheorem{lemma}{Lemma}[section]
\newtheorem{proposition}{Proposition}[section]
\theorembodyfont{\upshape}
\newtheorem{definition}{Definition}
\newtheorem*{remark}{Remark}
\newtheorem{example}{Example}[section]

\makeatletter
\let\c@proposition\c@theorem
\let\c@lemma\c@theorem
\let\c@corollary\c@theorem
\makeatother

\newenvironment{proof}{\noindent{\sc Proof:}}{\vspace{-0.5cm}~\hfill $\square$\vspace{0.5cm}}
\newenvironment{nproof}[1]{\noindent{\sc Proof #1:}}{\vspace{-1em}~\hfill $\square$\vspace{2em}}

\newcommand\N{\mathbb{N}}
\newcommand\R{\mathbb{R}}
\newcommand\Z{\mathbb{Z}}
\newcommand\E{\mathbb{E}\,}
\newcommand\Prob{\mathbb{P}}
\renewcommand\epsilon{\varepsilon}
\renewcommand\phi{\varphi}

\definecolor{darkblue}{rgb}{0,0,.5}
\hypersetup{colorlinks=true, breaklinks=true, linkcolor=blue, 
citecolor=darkblue, menucolor=blue, urlcolor=blue}

\begin{document}
\newpage
\maketitle
\begin{abstract}
    If the nodes of a graph are considered to be identical barrels -- featuring different water
    levels -- and the edges to be (locked) water-filled pipes in between the barrels, one might
    consider the optimization problem of how much the water level in a fixed barrel can be raised
    with no pumps available, i.e.\ by opening and closing the locks in an elaborate succession.
    This problem originated from the analysis of an opinion formation process and proved to be not
    only sufficiently intricate in order to be of independent interest, but also algorithmically
    complex. We deal with both finite and infinite graphs as well as deterministic and random
    initial water levels and find that the infinite line graph, due to its leanness, behaves much
    more like a finite graph in this respect.
\end{abstract}



\section{Introduction}

Imagine a plane on which rainwater is collected in identical rain barrels, some of which are connected through
pipes (that are already water-filled). All the pipes feature locks that are normally closed. If a lock is opened,
the contents of the two barrels which are connected via this pipe start to level, see Figure \ref{barrels}.
If one waits long enough, the water levels in the two barrels will be exactly the same, namely lie at the average
$\tfrac{a+b}{2}$ of the two water levels ($a$ and $b$) before the pipe was unlocked.

\begin{figure}[ht]
     \centering
     \includegraphics[scale=0.9]{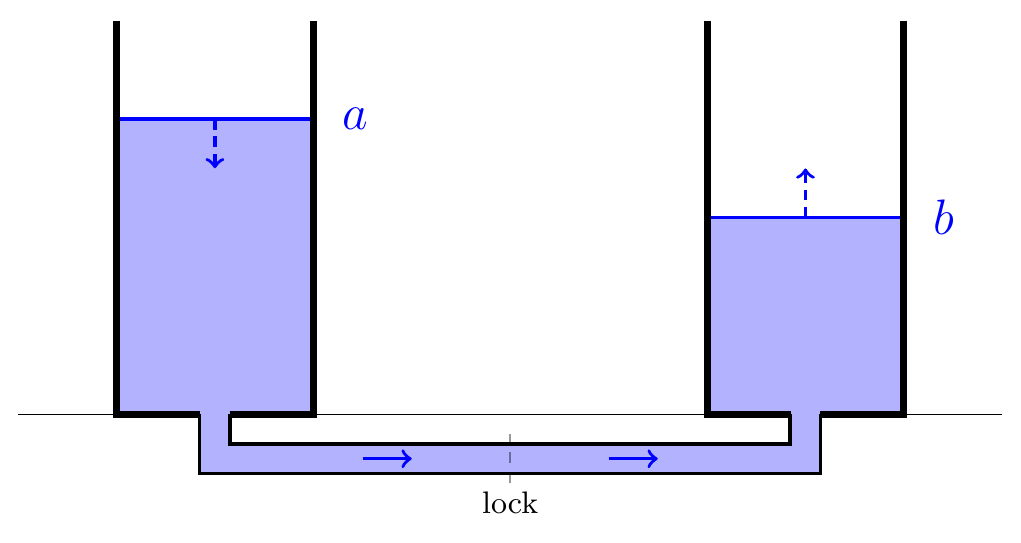}
     \caption{Levelling water stages after just having opened a lock. \label{barrels}}
\end{figure}

After a rainy night in which all of the barrels accumulated a certain amount of precipitation we might be interested
in maximizing the water level in one fixed barrel by opening and closing some of the locks in carefully chosen
order.

\vspace*{1em}
\noindent 
In order to mathematically model the setting, consider an undirected graph $G=(V,E)$, which is either finite or
infinite with bounded maximal degree. Furthermore, we can assume without loss of generality that $G$ is
connected and simple, that means having neither loops nor multiple edges. Every vertex is understood to represent
one of the barrels and the pipes correspond to the edges in the graph. The barrels themselves are considered
to be identical, having a fixed capacity $C>0$.

Given some initial profile $\{\eta_0(u)\}_{u\in V}\in[0,C]^V$, the system is considered to evolve in discrete
time and in each round we can open one of the locked pipes and
transport water from the fuller barrel into the emptier one. If we stop early, the two levels might not have
completely balanced out giving rise to the following update rule for the water profile: If in round $k$ the pipe
$e=\langle x,y\rangle$ connecting the two barrels at sites $x$ and $y$, with levels $\eta_{k-1}(x)=a$ and
$\eta_{k-1}(y)=b$ respectively, is opened and closed after a certain period of time, we get

 \begin{equation}\label{update}
         \begin{array}{rcl}\eta_k(x) &\!=\!& a+\mu_k\,(b-a)\\
                           \eta_k(y) &\!=\!& b+\mu_k\,(a-b)
         \end{array}
 \end{equation}
for some $\mu_k\in[0,\tfrac12]$, which we assume can be chosen freely by appropriate choice of how long the pipe
is left open. All other levels stay unchanged, i.e.\ $\eta_k(w)=\eta_{k-1}(w)$ for all $w\in V\setminus\{x,y\}$.

The quantity of interest is then defined as follows:

\begin{definition}\label{kappa}
	For a graph $G=(V,E)$, an initial water profile $\{\eta_0(u)\}_{u\in V}$ and a fixed vertex $v\in V$
	(the {\em target vertex}), let a {\em move sequence} be given by a list of edges and time spans
	that determines which pipes are opened (in chronological order) and for how long.
	Let then $\kappa(v)$ be defined as the supremum over all water levels that are
	achievable at $v$ with move sequences consisting of finitely many rounds, i.e.\
	$$\kappa(v):=\sup\{r\in\R,\, \text{there exists }T\in\N_0 \text{ and a move sequence s.t.\ }\eta_T(v)=r\}.$$
\end{definition}

Readers familiar with mathematical models for social interaction processes might note that (\ref{update})
basically looks like the update rule in the opinion formation process given
by the so-called {\em Deffuant model} for consensus formation in social networks (as described in the
introduction of \cite{Deffuant}), only $\mu$ can change from update to update and the bounded confidence
restriction is omitted. This however is no coincidence: The situation described in the context above arises
naturally in the analysis of the Deffuant model where the question is how extreme an opinion can a fixed
agent possibly get given an initial opinion profile on a specified network graph, if the interactions take
place appropriately.

In order to tackle this question, Häggström \cite{ShareDrink} invented a non-random pairwise averaging
procedure, which he proposed to call {\em Sharing a drink} (SAD). This procedure -- which is the main focus
of the second section -- was originally considered on the infinite line graph only, i.e.\ the graph $G=(V,E)$
with $V=\Z$ and $E=\{\langle v, v+1\rangle,\;v\in\Z\}$, but can immediately be generalized to any
graph (see Definition \ref{SAD}) and is dual to the water transport described above in a sense to be made
precise in Lemma \ref{dual}.

In Section \ref{finite}, we will deal with the water transport problem on finite graphs. After formally
introducing the idea of optimal move sequences, we investigate both their essential building blocks and
the effect of simultaneously opened pipes. 
In subsection \ref{examples}, being a collection of examples, we will in fact deal with both situations -- the
one in which we consider the initial water levels to be deterministic and the other in which they are random.
In the latter case $\kappa(v)$ obviously becomes a random variable as well since it strongly depends on the
initial profile. On non-transitive graphs (see Definition \ref{quasitrans}) its distribution can moreover
depend on the chosen vertex $v$ -- even for i.i.d.\ initial water levels, see Example \ref{linegraph}.

In the fourth section, we extend the complexity consideration touched upon in some of the examples
from Section \ref{finite}. We show that it is an NP-hard problem to determine $\kappa(v)$ for a
given finite graph, target vertex $v$ and initial water profile in general, something that might be considered
as a valid excuse for the fact that we are unable to give a neat general solution when it comes to
optimal move sequences in the water transport problem on finite graphs, as dealt with in Section \ref{finite}.

As opposed to the two precedent sections, Section \ref{infinite} is devoted to infinite graphs. We consider
i.i.d.\ initial water levels (with a non-degenerate marginal distribution) and detect a remarkable change
of behavior: On the infinite line graph, the highest achievable water level at a fixed vertex depends on
the initial profile in the sense that is has a non-degenerate distribution, just like on any finite graph.
If the infinite graph contains a neighbor-rich half-line (see Definition \ref{hl}), however, this dependence
becomes degenerate: For any vertex $v\in V$, the value $\kappa(v)$ almost surely equals the essential
supremum of the marginal distribution. This fact makes the infinite line graph quite unique:
It constitutes the only exception among all infinite quasi-transitive graphs, to the effect that $\kappa(v)$ is a
non-degenerate random variable -- an observation which is captured in the last theorem: the nonetheless central 
Theorem \ref{qtgraphs}.

\section{Connection to the SAD-procedure}\label{sec2}

Let us first repeat the formal definition of the SAD-procedure:
\begin{definition}\label{SAD}
For a graph $G=(V,E)$ and some fixed vertex $v\in V$, define $\{\xi_0(u)\}_{u\in V}$ by setting
$$\xi_0(u)=\begin{cases}1\quad\text{for } u=v\\ 0 \quad\text{for } u\neq v.\end{cases}$$
In each time step, an edge $\langle x,y\rangle$ is chosen and the profile $\{\xi_0(u)\}_{u\in V}$ updated
according to the rule (\ref{update}) with $\{\xi_k(u)\}_{u\in V}$ in place of $\{\eta_k(u)\}_{u\in V}$.
One can interpret this procedure as a full glass of water initially placed at vertex $v$ (all other glasses
being empty), which is then repeatedly shared among neighboring vertices by each time step choosing a pair
of neighbors and pouring a $\mu_k$-fraction of the difference from the glass containing more water into the one
containing less. Let us refer to this interaction process as {\em Sharing a drink (SAD)}.
\end{definition}

Just as in \cite{ShareDrink}, the SAD-procedure can be used to describe the composition of the contents
in the water barrels after finitely many rounds of opening and closing pipe locks. The following lemma
corresponds to La.\ 3.1 in \cite{ShareDrink}, but since the two dual processes (water transport and SAD) evolve
in discrete time in our setting, the proof simplifies somewhat.

\begin{lemma}\label{dual}
Consider an initial profile of water levels $\{\eta_0(u)\}_{u\in V}$ on a graph $G=(V,E)$ and fix a
vertex $v\in V$. For $T\in\N_0$ define the SAD-procedure that starts with $\xi_0(u)=\delta_v(u)$
(see Definition \ref{SAD}) and is dual to the chosen move sequence in the water transport problem in the 
following sense: If in round $k\in\{1,\dots,T\}$ the water profile is updated according to (\ref{update}), the update in the SAD-profile at
time $T-k\in\{0,\dots,T-1\}$ takes place along the same edge and with the same choice of $\mu_k$. Then we get
\begin{equation}\label{convcomb}\eta_T(v)=\sum_{u\in V}\xi_T(u)\,\eta_0(u).\end{equation}
\end{lemma}

\begin{proof}
We prove the statement by induction on $T$. For $T=0$, the statement is trivial and there is nothing to show.
For the induction step fix $T\in\N$ and assume the first pipe opened to be $e=\langle x,y\rangle$. According to 
(\ref{update}) we get
$$\eta_1(u)=\begin{cases}\eta_0(u)& \text{if }u\notin\{x,y\}\\
                         (1-\mu_1)\,\eta_0(x)+\mu_1\,\eta_0(y)&\text{if }u=x\\
                         (1-\mu_1)\,\eta_0(y)+\mu_1\,\eta_0(x)&\text{if }u=y.
                         \end{cases}$$
Let us consider $\{\eta_1(u)\}_{u\in V}$ as some initial profile $\{\eta'_0(u)\}_{u\in V}$. By induction
hypothesis we get
\begin{align*}
\eta'_{T-1}(v)&=\sum_{u\in V}\xi'_{T-1}(u)\,\eta'_0(u)\\
            &=\sum_{u\in V\setminus\{x,y\}}\xi'_{T-1}(u)\,\eta_0(u)
            +\Big((1-\mu_1)\,\xi'_{T-1}(x)+\mu_1\,\xi'_{T-1}(y)\Big)\,\eta_0(x)\\
            &\quad+\Big((1-\mu_1)\,\xi'_{T-1}(y)+\mu_1\,\xi'_{T-1}(x)\Big)\,\eta_0(y),
\end{align*}
where $\eta'_{T-1}(v)=\eta_T(v)$ and $\{\xi'_k(u)\}_{u\in V}$, $0\leq t\leq T-1$, is the SAD-procedure corresponding
to the move sequence after round $1$. As by definition the SAD-procedure $\xi$ arises from $\xi'$ by
adding an update at time $T-1$ along edge $e$ with parameter $\mu_1$, we find $\xi_k(u)=\xi'_k(u)$ for all 
$k\in\{0,\dots,T-1\}$ and $u\in V$ as well as
$$\xi_T(u)=\begin{cases}\xi_{T-1}(u)=\xi'_{T-1}(u)& \text{if }u\notin\{x,y\}\\
                         (1-\mu_1)\,\xi_{T-1}(x)+\mu_1\,\xi_{T-1}(y)&\text{if }u=x\\
                         (1-\mu_1)\,\xi_{T-1}(y)+\mu_1\,\xi_{T-1}(x)&\text{if }u=y,
                         \end{cases}$$
which establishes the claim.
\end{proof}

\vspace*{1em}
In the following sections, we want to consider not only deterministic but also random initial profiles of
water levels. Having this mindset already, it might be useful to halt for a moment and realize that the statement
of Lemma \ref{dual} deals with a deterministic duality that does not involve any randomness (once the initial
profile and the move sequence are fixed).

Before we turn to the task of rising water levels, let us prepare two more auxiliary results. The first one
follows directly from the energy argument that was used in the proof of Thm.\ 2.3 in \cite{ShareDrink}:
\begin{lemma}\label{evenout}
Given an initial profile of water levels $\{\eta_0(u)\}_{u\in V}$ on a graph $G=(V,E)$, fix a finite set $A\subseteq V$
and a set $E_A\subseteq E$ of edges inside $A$ that connects $A$.
If we open the pipes in $E_A$ -- and no others -- in repetitive sweeps for times long enough
such that $\mu_k\geq\epsilon$ for some fixed $\epsilon>0$ in each round (cf.\ (\ref{update})),
then the water levels inside the set $A$ approach a balanced average, i.e.\ 
converge to the value $\tfrac{1}{|A|} \sum_{v\in A}\eta_0(v)$. The corresponding dual SAD-profiles started with
$\xi_0(u)=\delta_v(u),\ u\in V,$ converge uniformly to $\tfrac{1}{|A|}\,\delta_A$ for all $v\in A$.
\end{lemma}

\begin{proof}
 Let us define the energy after round $k$ inside $A$ by
 $$W_k(A)=\sum_{v\in A}\big(\eta_k(v)\big)^2.$$
 A short calculation reveals that an update of the form (\ref{update}) reduces the energy by
 $2\mu_k^{\,2}\,(b-a)^2$, where the updated water levels were $a$ and $b$ respectively. If $\mu_k$ is bounded away from $0$,
 the fact that $W_k(A)\geq0$ for all $k$ entails that the difference in water levels $|b-a|$ before a pipe is opened
 can be larger than any fixed positive value only finitely many times. In effect, since any pipe in $E_A$ is opened
 repetitively we must have $|\eta_k(u)-\eta_k(v)|\to 0$ as $k\to\infty$ for all edges $\langle u,v\rangle\in E_A$.
 As the updates are average preserving, the first part of the claim follows from the fact that $E_A$ connects $A$.
 
 The second part of the lemma follows by applying the same argument to the dual SAD-procedure.
 \end{proof}
 \vspace*{1em}

The following lemma constitutes an extremely narrowed variant of Thm.\ 2.3 in \cite{ShareDrink} which applies to
graphs more complex than line graphs as well and will come in useful in Example \ref{complete}:
\begin{lemma}\label{SADmax}
Fix a (connected) graph $G=(V,E)$ and a vertex $v\in V$. For any $w\in V\setminus{\{v\}}$,
the supremum of $\xi_k(w)$ taken over all times $k$ and SAD-procedures started with 
$\xi_0(u)=\delta_v(u)$, $u\in V$, is less than or equal to $\tfrac{1}{2}$.
\end{lemma}

\begin{proof}
	If the SAD-procedure is started with a full glass of water at $v\neq w$, the assumption that the amount
	at $w$ can rise above $\tfrac12$ leads to the following contradiction: Assume $k$ to be the first time s.t.\ 
	$\xi_k(w)>\tfrac12$. Then in round $k$ node $w$ necessarily pooled the water with some neighbor $u$, that had more
	water than $w$. But since this relation is preserved by an update, it implies
	$$\xi_k(w)+\xi_k(u)\geq2\,\xi_k(w)>1,$$
	which is impossible as the amount of water shared always sums to 1.
\end{proof}\vspace*{1em}
 
 To round off these preliminary considerations, let us collect some results about SAD-profiles from $\cite{ShareDrink}$
 -- partly already mentioned -- into a single lemma for convenience.
 
 \begin{lemma}\label{collection}
  Consider the SAD-procedure on a line graph, started in vertex $v$, i.e.\ with $\xi_0(u)=\delta_v(u),\ u\in V$.
  \begin{enumerate}[(a)]
   \item The SAD-profiles achievable on line graphs are all unimodal.
   \item If the vertex $v$ only shares the water to one side, it will remain a mode of the SAD-profile.
   \item The supremum over all achievable SAD-profiles started with $\delta_v$ at another vertex $w$ equals $\tfrac{1}{d+1}$,
         where $d$ is the graph distance between $v$ and $w$.
 \end{enumerate}
 \end{lemma}	
 
  The results in \cite{ShareDrink} actually all deal with the infinite line graph, but it is evident how the arguments
  used immediately transfer to finite line graphs. Part (a) hereby corresponds to La.\ 2.2 in \cite{ShareDrink}, part
  (b) to La.\ 2.1 and part (c) to Thm.\ 2.3.
  The argument Häggström \cite{ShareDrink} used to prove the statement in (c) for the infinite line graph can in fact
  be generalized to prove the result for trees without much effort, as was done by Shang (see Prop.\ 6 in \cite{Shang}). 
  
In fact, we believe that not only the cut back statement from Lemma \ref{SADmax} but also the natural generalization
of Thm.\ 2.3 in \cite{ShareDrink} holds true for general graphs. Our attempts to prove the generalization to non-tree
graphs have, however, turned out unsuccessful.

\section{Water transport on finite graphs}\label{finite}

In this section, we consider the underlying network to be finite, i.e.\ $|V|=n\in\N$.
In order to increase the water level at our fixed site $v$ one could in principle start by
greedily trying to connect the barrels with the highest water levels to the one at $v$.
However, optimizing this idea is far from being trivial. Let us first define optimal move sequences and
then reveal some properties and building blocks that they share.

\begin{definition}\label{opt}
For fixed $v\in V$ and a given initial water profile $\{\eta_0(u)\}_{u\in V}$ let $\phi\in (E\times[0,\tfrac12])^T$,
where $\phi_k=(e_k,\mu_k)$, be called a {\em finite move sequence} if $T\in\N_0$. $\phi$ is a {\em finite optimal
move sequence} if opening the pipes $e_1,\dots,e_T$ in chronological order, each for the period of time that corresponds
to $\mu_k$ in (\ref{update}), will lead to the final value $\eta_T(v)=\kappa(v)$.

For any move sequence $\phi\in (E\times[0,\tfrac12])^T$, we will denote by $\{\xi_T(u)\}_{u\in V}$ the SAD-profile that
corresponds to $\phi$ via the duality laid down in Lemma \ref{dual}.

If no finite optimal move sequence exists, let us call $\Phi=\{\phi^{(m)},\;m\in\N\}$ an {\em  infinite type
optimal move sequence}, provided that $\phi^{(m)}\in (E\times[0,\tfrac12])^{T_m}$ is a finite move sequence for each $m\in\N$,
achieving $\eta_{T_m}(v)> \kappa(v)-\tfrac1m$ and the SAD-profiles $\{\xi_{T_m}(u)\}_{u\in V}$ dual to $\phi^{(m)}$
converge pointwise to a limit $\{\xi(u)\}_{u\in V}$ as $m\to\infty$.
\end{definition}

It is tempting to assume that in the case where no finite optimal move sequence exists, we could get away with an
infinite move sequence instead of a sequence of finite move sequences $\Phi$ as described above. However this is not the
case, see Example \ref{seqofseq}.

\begin{lemma}\label{simplif}
	Take the network $G=(V,E)$ to be finite, and fix the target vertex $v$ as well as the initial water profile.
	Then the existence of an optimal move sequence is guaranteed and the following simplification will not
	change its performance:
	In an optimal move sequence, without loss of generality we can assume $\mu_k=\tfrac12$ for all $k$.
\end{lemma}

\begin{proof}
By the very definition of $\kappa(v)$, the existence of optimal move sequences (however not necessarily finite ones)
is guaranteed: Let $A\subseteq[0,1]^V$ denote the set of achievable SAD-profiles. Its closure $\overline{A}$ in 
$([0,1]^V, \lVert\,.\,\rVert_2)$ is bounded and therefore compact. Given the initial water profile
$\{\eta_0(u)\}_{u\in V}$, the function
$$f:=\begin{cases}[0,1]^V\to[0,C]\\
                  \{\xi(u)\}_{u\in V}\mapsto \sum_{u\in V}\limits\xi(u)\,\eta_0(u)\end{cases}$$
is continuous. Hence there exists a closed subset $F$ of $\overline{A}$ on which $f$ achieves its maximum $\kappa(v)$
over $\overline{A}$. The SAD-profiles dual to finite optimal move sequences are given by $F\cap A$. If $F\cap A=\emptyset$
and $\Phi=\{\phi^{(m)},\;m\in\N\}$ is a collection of finite move sequences s.t.\ $\phi^{(m)}\in (E\times[0,\tfrac12])^{T_m}$ and
$\eta_{T_m}(v)> \kappa(v)-\tfrac1m$ for all $m\in\N$, we can assume without loss of generality that
the corresponding SAD-profiles $\{\xi_{T_m}(u)\}_{u\in V}$ have a limit $\{\xi(u)\}_{u\in V}$ (by passing on to
a subsequence if necessary) as $\overline{A}$ is compact. This turns $\Phi$ into an infinite type optimal move sequence and
the limit of its dual SAD-profiles necessarily lies in $F$.
                  
Assume now that the first move in a sequence $\phi\in (E\times[0,\tfrac12])^T$ is to open the lock on pipe
$e_1=\langle x,y\rangle$ for a time corresponding to $\mu_1\in[0,\tfrac12]$ in (\ref{update}).
       
Without loss of generality we can assume $\eta_0(x)\geq\eta_0(y)$ (which in turn implies
$\eta_1(x)\geq\eta_1(y)$). If we look at the SAD-profile $\{\xi'_{T-1}(u)\}_{u\in V}$ corresponding to
$\phi':=(\phi_2,\dots,\phi_T)\in (E\times[0,\tfrac12])^{T-1}$ -- in effect we look at the outcome of the move sequence
after the first step applied to the new initial water profile $\{\eta_1(u)\}_{u\in V}$ -- we can distinguish two cases:
either $\xi'_{T-1}(x)\geq\xi'_{T-1}(y)$ or $\xi'_{T-1}(x)<\xi'_{T-1}(y)$.
In the first case changing $\mu_1$ to $0$, i.e.\ erasing the first move will not decrease the water level
finally achieved at $v$, see (\ref{convcomb}). In the second case the same holds for changing $\mu_1$ to
$\tfrac12$. Since we can consider any step in the move sequence to be the first one applied to the intermediate
water profile achieved so far, this establishes the claim for finite optimal move sequences.

As any finite move sequence can be simplified in this way without worsening its outcome, the argument applies to the elements
of a sequence $\Phi=\{\phi^{(m)},\;m\in\N\}$ of finite move sequences and thus to infinite type optimal move sequences
as well.
\end{proof}

\subsection{Macro moves}\label{macro}

When it comes to the opening and closing of pipes, it is not self-evident how far things change if we allow
pipes to be opened simultaneously. First of all one has to properly extend the model laid down in (\ref{update})
by specifying how the water levels behave when more than two barrels are connected at the same time. In order to
keep things simple, let us assume that the pipes are all short enough and of sufficient diameter such that we can
neglect all kinds of flow effects. Moreover, let us take the dynamics to be as crude as can be by assuming that
the water levels of the involved barrels approach their common average in a linear and proportional fashion, which
is made more precise in the following definition.

\begin{definition}\label{macro_def}
Given a graph $G=(V,E)$, let $A\subseteq V$ be a set of at least 3 nodes and $E_A\subseteq E$ a set of edges inside
$A$ that connects $A$. A {\em macro move} on $E_A$ (or simply $A$) will denote the action of opening all pipes that
correspond to edges in $E_A$ in some round $k$ simultaneously and will -- analogously to (\ref{update}) -- change
the water levels for all vertices $u\in A$ to
$$\eta_k(u)=(1-2\mu_k)\,\eta_{k-1}(u)+2\mu_k\, \overline{\eta}_{k-1}(A),\quad \text{where }
\overline{\eta}_{k-1}(A)=\frac{1}{|A|}\,\sum_{w\in A}\eta_{k-1}(w)$$
is the average over the set $A$ after round $k-1$ and $\mu_k\in[0,\tfrac12]$.
\end{definition}

First of all, Lemma \ref{dual} transfers immediately and almost verbatim to move sequences including macro moves:
In a move sequence with a macro move on the set $A$ in the first round, we get the water levels
$$\eta_1(u)=\begin{cases}\eta_0(u)& \text{if }u\notin A\\
(1-2\mu_1)\,\eta_0(u)+2\mu_1\,\overline{\eta}_0(A)&\text{if }u\in A.
\end{cases}$$
If $\{\xi_{T-1}(u),\;u\in V\}$ and $\{\xi_{T}(u),\;u\in V\}$ are such that
$$\eta_T(v)=\sum_{u\in V}\xi_T(u)\,\eta_0(u)=\sum_{u\in V}\xi_{T-1}(u)\,\eta_1(u),$$
we find by comparing the coefficient of $\eta_0(u)$
$$\xi_T(u)=\begin{cases}\xi_{T-1}(u)& \text{if }u\notin A\\
(1-2\mu_1)\,\xi_{T-1}(u)+\sum_{w\in A} 2\mu_1\,\frac{\xi_{T-1}(w)}{|A|}&\text{if }u\in A,
\end{cases}$$
which is the SAD-profile originating from the very same macro move applied to $\{\xi_{T-1}(u),\;u\in V\}$.
With this tool in hand, we can prove the following extension of Lemma \ref{simplif}:

\begin{lemma}\label{simplif2}
Take the network $G=(V,E)$ to be finite, and fix the target vertex $v$ as well as the initial water profile.
\begin{enumerate}[(a)]
 \item Even if we allow macro moves, the statement of Lemma \ref{simplif} still holds true, i.e.\ reducing the range
       of $\mu_k$ from $[0,\tfrac12]$ to $\{0,\tfrac12\}$ in each round $k$ does not worsen the outcome of optimal
       move sequences.
 \item The sharp upper bounds on achievable water levels are not changed if we allow for pipes to be opened 
       simultaneously. In other words, the supremum $\kappa(v)$ of water levels achievable at a vertex $v$, as
       characterized in Definition \ref{kappa}, stays unchanged if we allow move sequences to include macro moves.
 \end{enumerate}
\end{lemma} 
 
\begin{proof}
\begin{enumerate}[(a)]
	\item Just as in Lemma \ref{simplif}, we consider a move sequence consisting of finitely many (macro) moves 
	-- say again $T\in\N$ -- and especially the SAD-profile dual to the moves after round 1, denoted by 
	$\{\xi_{T-1}(u),\;u\in V\}$. If the first action is a macro move on the set $A$, let us divide its nodes into
	two subsets according to whether their initial water level is above or below the initial average across $A$:
	$$A_a:=\{u\in A,\;\eta_0(u)\geq\overline{\eta}_0(A)\} \quad\text{and}\quad
    A_b:=\{u\in A,\;\eta_0(u)<\overline{\eta}_0(A)\}.$$
    If $\sum_{u\in A_a}\xi_{T-1}(u)\leq\sum_{u\in A_b}\xi_{T-1}(u)$, changing $\mu_1$ to $\tfrac12$ will not
    decrease the final water level achieved at $v$. If instead $\sum_{u\in A_a}\xi_{T-1}(u)\geq\sum_{u\in A_b}\xi_{T-1}(u)$,
    the same holds for erasing the first move (i.e.\ setting $\mu_1=0$).
    
	\item Obviously, allowing for pipes to be opened simultaneously can if anything increase the maximal water level
	achievable at $v$. However, any such macro move can be at least approximated by opening
	pipes one after another. Levelling out the water profile on a set of more than 2 vertices completely will correspond
	to the limit of infinitely many single pipe moves on the edges between them (in a sensible order).
	
	Let us consider a finite move sequence $\phi$ including macro moves on the sets $A_1\,\dots,A_l$ (in chronological
	order). From part (a) we know that with regard to the final water level achievable at $v$ we can assume w.l.o.g.\
	that all moves are complete averages (i.e.\ $\mu_k=\tfrac12$ for all $k$). Fix $\epsilon>0$ and let us define a
	finite move sequence $\overline{\phi}$ including no macro moves in the following way:
	We keep all the rounds in $\phi$ in which pipes are opened individually. For the macro move on $A_i,\ i\in\{1,\dots,l\},$
	we insert a finite number of rounds in which the pipes of an edge set $E_{A_i}$ connecting $A_i$ are opened in
	repetitive sweeps such that the water level at each vertex $u\in A_i$ is less than $\tfrac{\epsilon}{2^i}$ away
	from the average across $A_i$ after these rounds. Note that Lemma \ref{evenout} guarantees that this is possible.
	
	As opening pipes leads to new water levels being convex combinations of the ones before, the differences of
	individual water levels	caused by replacing the macro moves add up to $\sum_{i=1}^l \tfrac{\epsilon}{2^i}<\epsilon$
	in the worst case. Consequently, the final water level achieved at $v$ by $\overline{\phi}$ is at most
	$\epsilon$ less than the one achieved by $\phi$. Since $\epsilon>0$ was arbitrary, this proves the claim.
	
	Note however that the option of macro moves can make a difference when it comes to the attainability of $\kappa(v)$,
	see Example \ref{seqofseq}.
\end{enumerate}
\end{proof}

\begin{remark}
	Lemma \ref{simplif2}\,(a) states that even for macro moves, there is nothing to be gained by closing the pipes before
	the water levels have balanced out completely. A macro move on the edge set $E_A$ with $\mu_k=\tfrac12$ can be seen
	as the limit of infinitely many single edge moves on $E_A$ in the sense of Lemma \ref{evenout} -- a connection that
	does not exist for macro moves with $\mu_k\in(0,\tfrac12)$.
	We believe that there always exists a finite optimal move sequence if macro moves are allowed. 
	We state this as an open problem.
\end{remark}

Due to Lemmas \ref{simplif} and \ref{simplif2} we can assume w.l.o.g.\ that the parameters $\mu_k$ in optimal move sequences
are always equal to $\tfrac12$ in each round, hence omit them and consider the move sequence to be a list of pipes
(i.e.\  $\phi\in E^T$) only.
We can incorporate an optimal move sequence in which more than one pipe is opened at a time into Definition
\ref{opt} by either allowing $\phi_k$, for $k\in\{0,\dots,T\}$, to be a subset of $E$ with more than one
element on which the levelling takes place or by viewing $\phi$ as a limiting case of move sequences
$\{\phi^{(m)},\;m\in\N\}$, in which pipes are opened separately, that form an infinite type optimal move sequence $\Phi$
-- as just described in the proof of the lemma.

In the sequel however -- if not otherwise stated -- we will stick to the initial regime where pipes are opened
one at a time.

\subsection{Optimizing the move sequence}\label{optimization}

Closely related to the water transport idea is the concept of {\em greedy lattice animals} as introduced by
Cox, Gandolfi, Griffin and Kesten \cite{GLA1}. The vertices of a given graph $G$ are associated with an i.i.d.\ 
sequence of non-negative random variables and a greedy lattice animal of size $n$ is then defined to be a 
connected subset of $n$ vertices containing the target vertex $v$ and maximizing the sum over the associated $n$
random variables. Since we do not care about the size of the lattice animal, let us slightly change this
definition:

\begin{definition}\label{GLAdef}
For a fixed graph $G=(V,E)$, target vertex $v$ and water levels $\{\eta(u)\}_{u\in V}$, let us call
$C\subseteq V$ a {\em lattice animal (LA)} for $v$ if $C$ is connected and contains $v$. $C$ is a {\em greedy lattice animal (GLA)} for $v$ if it maximizes the
average of water levels over such sets.
This average will be considered as its value
$$\text{GLA}(v):=\frac{1}{|C|}\sum_{u\in C}\eta(u).$$
\end{definition}

By Lemma \ref{evenout}, it is clear that $\text{GLA}(v)\leq\kappa(v)$. In fact, for the majority of settings
-- consisting of a graph $G$, a target vertex $v$ and an initial water profile $\{\eta_0(u)\}_{u\in V}$ -- 
strict inequality holds and we can do better than just pooling the amount of water collected in an appropriately
chosen connected set of barrels including the one at $v$.

Furthermore we know from Lemma \ref{simplif2}\,(a) that w.l.o.g.\ the last move of any finite optimal move sequence will
be to pool the amount of water allocated in a connected set of vertices including $v$. This greedy lattice
animal for $v$ in the intermediate water profile created up to that point in time can be more advantageous than
the one in the initial water profile if we imply the following improving steps first:

\begin{description}
  \item[1) Improving bottlenecks] \hfill \\
  Let us call a vertex $u$ a {\em bottleneck} of the GLA $C$ for $v$ if $u\in C\setminus\{v\}$ and
  $\eta(u)<\text{GLA}(v)$.
  Clearly, each bottleneck $u$ has to be a cutting vertex for $C$ (otherwise we could just remove it to
  improve the GLA). If there exists a connected subset of vertices $C_u$ including $u$ which has a higher
  average water level than $C_u\cap C$, the value of the GLA for $v$ is improved if the water collected in
  $C_u$ is pooled first (see Figure \ref{GLA1}). Note that $C_u$ might involve more vertices from $C$ than
  just $u$, see Example \ref{linegraph3}.
  \item[2) Enlargement] \hfill \\
  The second option to raise the value of the GLA $C$ for $v$ is to apply the idea above to a vertex $u$ in
  the vertex boundary of $C$ in order for the original GLA to be enlarged to a set of vertices in which
  $u$ is a bottleneck. For this to be beneficial, there has to exist a connected set of vertices $C_u$ in
  $V\setminus C$ including $u$ with the following property:
  The average water level in $C_u$ is smaller than $\text{GLA}(v)$ -- otherwise it would be part of $C$ --
  but is raised above this value after improving the potential bottleneck $u$ using water located in
  $V\setminus C$ (see Figure \ref{GLA1} below).
  
  \begin{figure}[H]
     \centering
     \includegraphics[scale=0.9]{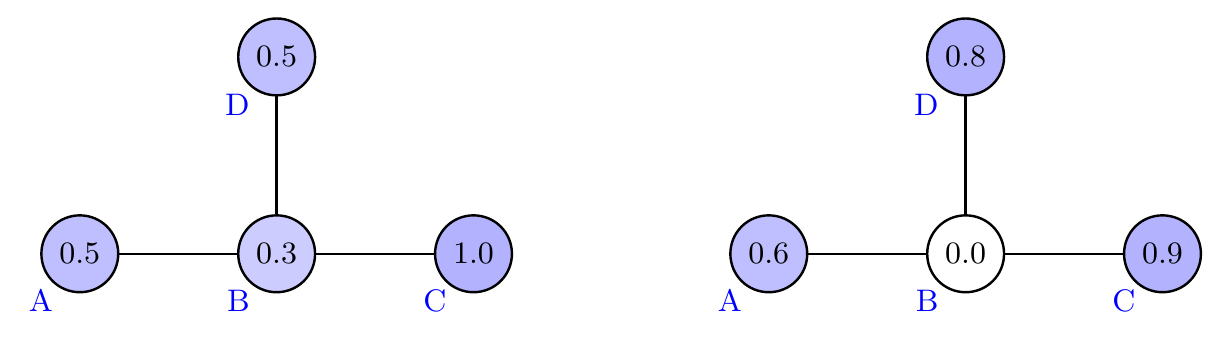}
     \caption{If $A$ is the target vertex, the GLA on the left is $\{A,B,C\}$ (having value 0.6) and the
     bottleneck $B$ can be improved by first opening the pipe $\langle B,D\rangle$.\newline     
     The GLA for $A$ with respect to the water profile on the right is $\{A\}$, but can be enlarged
     to $\{A,B,C\}$ if the potential bottleneck $B$ is improved by opening the pipe
     $\langle B,D\rangle$ first. \label{GLA1}}
   \end{figure}
  
  \item[3) Choose optimal chronological order] \hfill \\
  When applying the improving techniques just described, it is essential to choose the optimal chronological
  order of doing things. Besides the fact that improving bottlenecks and enlarging the GLA has to be done
  before the final averaging, situations can arise in which different sets of vertices can improve the
  same bottleneck or the other way round that more than one bottleneck can be improved using non-disjoint
  sets of vertices, see the set-ups in Figure \ref{GLA2}.
  
  \begin{figure}[H]
     \centering
     \includegraphics[scale=0.9]{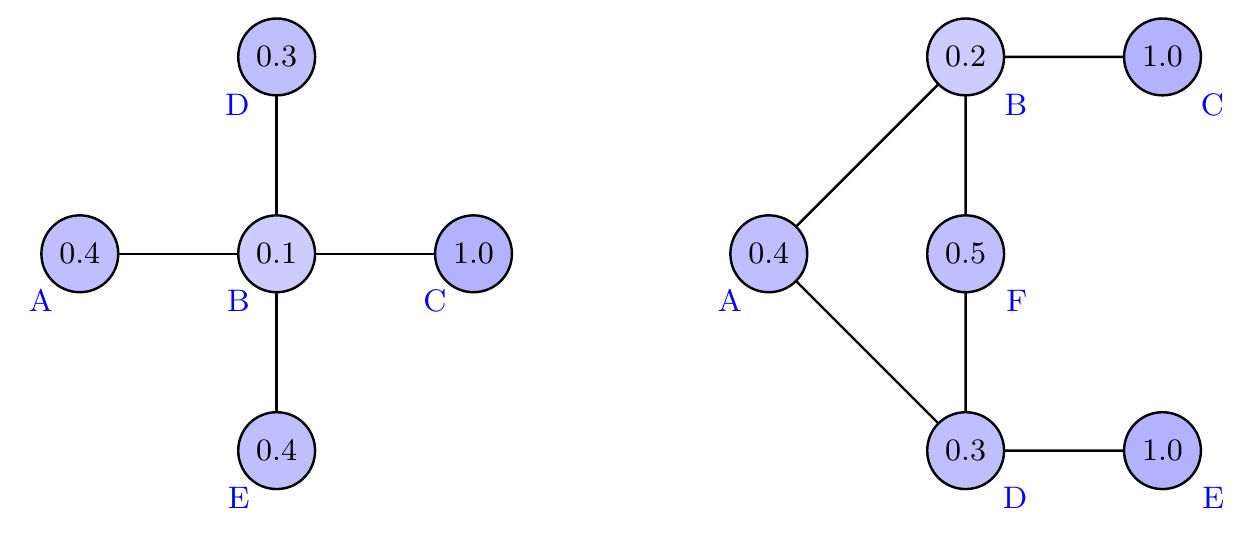}
     \caption{If $A$ is the target vertex, the GLA on the left is $\{A,B,C\}$ (having value 0.5).
     Improving the bottleneck $B$ can be done using $D$ or $E$ and is most effective if the pipe
     $\langle B,D\rangle$ is opened first, then $\langle B,E\rangle$.\newline     
     The GLA for $A$ with respect to the graph on the right is $\{A,B,C,D,E\}$. The water from $F$
     can be used to improve both bottlenecks $B$ and $D$. It is optimal to open pipe $\langle D,F\rangle$
     first and then $\langle B,F\rangle$. \label{GLA2}}
   \end{figure}
\end{description}
\noindent
  Finally, it is worth noticing that lattice animals with lower average than $\text{GLA}(v)$ in the initial
  water profile sometimes can be improved by the techniques just described to finally outperform the initial
  GLA and its possible improvements and enlargements (see Example \ref{linegraph3} and especially Figure
  \ref{15-line2}).

\subsection{Examples}\label{examples}

\begin{example}
The minimal graph which is non-trivial with respect to water transport is a single edge, in other words
the complete graph on two vertices:
$$G=K_2=(\{1,2\},\{\langle1,2\rangle\}).$$
By the considerations in the previous subsection, we get
 \begin{equation}\label{edgekappa}
  \kappa(1)=\begin{cases}\eta_0(1)&\text{if }\eta_0(1)\geq\eta_0(2)\\
                        \tfrac{\eta_0(1)+\eta_0(2)}{2}&\text{if }\eta_0(1)<\eta_0(2).
           \end{cases}
 \end{equation}
Let the initial water levels be given by the two random variables $U_1$ and $U_2$. From (\ref{edgekappa})
it immediately follows that
$$U_1\leq \kappa(1)\leq \max\{U_1, U_2\}.$$

If we assume $U_1$ and $U_2$ to be independent and uniformly distributed on $[0,1]$, a short calculation
reveals the distribution function
$$F_{\kappa(1)}(x)=\begin{cases}\tfrac32 x^2&\text{for }0\leq x \leq \tfrac12\\
                            x-\tfrac12\,(1-x)^2&\text{for }\tfrac12\leq x \leq 1,
           \end{cases}$$
which indeed lies in between $F_{U_1}(x)=x$ and $F_{\max\{U_1,U_2\}}(x)=x^2$, see Figure \ref{plot}.
\begin{figure}[H]
     \centering
     \includegraphics[scale=0.75]{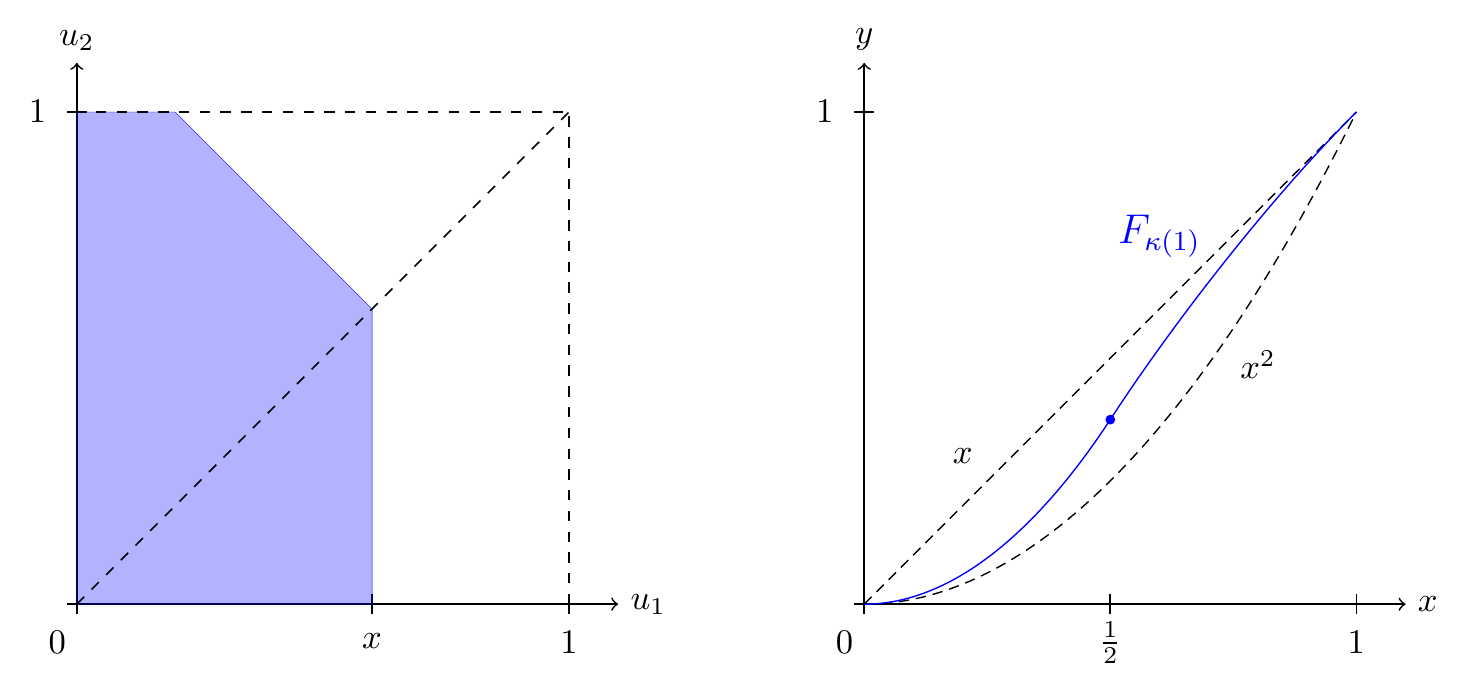}
     \caption{On the left a visualization of $\Prob(\kappa(1)\leq x)$, on the right the distribution
     function of $\kappa(1)$.\label{plot}}
\end{figure}
\noindent By symmetry, the exact same considerations hold for $\kappa(2)$.
\end{example}

\begin{example}\label{linegraph}
\par\begingroup \rightskip14em\noindent
The simplest non-transitive graph (i.e.\ having vertices of different kind, see Definition \ref{quasitrans})
is the line graph on three vertices:
\par\endgroup

\vspace*{-1.5cm}
   \begin{figure}[H]
     \flushright \includegraphics[scale=0.9]{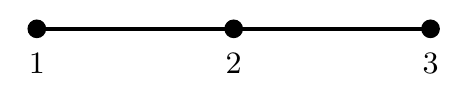}
   \end{figure}
\vspace*{-0.5cm}
 
$$G=(\{1,2,3\},\{\langle1,2\rangle,\langle2,3\rangle\}).$$
Again by the above considerations, we find the supremum of achievable water levels at vertex $1$ to be
\begin{equation*}
  \kappa(1)=\max\Big\{\eta_0(1),\tfrac{\eta_0(1)+\eta_0(2)}{2},\tfrac{\eta_0(1)+\eta_0(2)+\eta_0(3)}{3}\Big\},
 \end{equation*}
which is obviously achieved by a properly chosen greedy lattice animal.

Consider the case in which the initial water levels satisfy
\begin{equation}\label{nofin}
 \eta_0(3)\geq\eta_0(2)\geq\eta_0(1)\quad \text{and}\quad \eta_0(3)>\eta_0(1)
\end{equation}
Then $\kappa(1)=\tfrac{\eta_0(1)+\eta_0(2)+\eta_0(3)}{3}$ and there exists no finite optimal move sequence.
This can be seen from the fact that any single move will preserve the inequalities in (\ref{nofin}) and thus
we have $\eta_T(1)<\kappa(1)<\eta_T(3)$ for all finite move sequences $\phi\in E^T$.
 
If we consider the initial water levels to be independent and identically distributed, the (random) supremum
of achievable water levels at vertex 2 is stochastically larger than the one at vertex 1: As $\eta_0(1)$ and
$\eta_0(2)$ have the same distribution so do
$$\kappa(1)\quad \text{and}\quad \max\Big\{\eta_0(2),\tfrac{\eta_0(1)+\eta_0(2)}{2},\tfrac{\eta_0(1)+\eta_0(2)+\eta_0(3)}{3}\Big\}.$$
The latter is less than or equal to $\kappa(2)$. The maximal value achievable by greedy lattice animals at vertex 2 is
\begin{equation*} \text{GLA}(2)=\max\Big\{\eta_0(2),\tfrac{\eta_0(1)+\eta_0(2)}{2},\tfrac{\eta_0(2)+\eta_0(3)}{2},
\tfrac{\eta_0(1)+\eta_0(2)+\eta_0(3)}{3}\Big\}. \end{equation*}
The fact that we can average across one pipe at a time and choose the order of updates allows us to improve over
this and gives
\begin{equation}\label{kappa2} \kappa(2)=\max\Big\{\text{GLA}(2),\tfrac12\,\big(\eta_0(1)+\tfrac{\eta_0(2)+\eta_0(3)}{2}\big),
\tfrac12\,\big(\eta_0(3)+\tfrac{\eta_0(1)+\eta_0(2)}{2}\big)\Big\}. \end{equation}
To see this, we can take a closer look on the SAD-profiles that can be created by updates along the two
edges $\langle1,2\rangle$ and $\langle2,3\rangle$ starting from the initial profile $\xi_0=(0,1,0)$: After one
update -- depending on the chosen edge -- the profile is given by $\xi_1=(\tfrac12,\tfrac12,0)$ or $(0,\tfrac12,\tfrac12)$.
After the second step we end up with either $\xi_2=(\tfrac12,\tfrac14,\tfrac14)$ or $(\tfrac14,\tfrac14,\tfrac12)$.
All of the corresponding convex combinations appear in the right hand side of (\ref{kappa2}). By Lemma \ref{evenout},
we know that continuing like this will finally result in the limiting profile $(\tfrac13,\tfrac13,\tfrac13)$.
It is not hard to check that any sequence of two or more updates will lead to a monotonous SAD-profile with a
largest value of at most $\tfrac12$ at one and a smallest value of at least $\tfrac14$ at the other end.
For this reason, it can be written as a convex combination of $(0,1,0)$ plus either $(\tfrac12,\tfrac14,\tfrac14)$ 
and $(0,\tfrac12,\tfrac12)$ or $(\tfrac14,\tfrac14,\tfrac12)$ and $(\tfrac12,\tfrac12,0)$. Consequently, it cannot
correspond to a final water level at vertex 2 exceeding the value in (\ref{kappa2}).

By an elementary calculation, for independent $\text{unif}([0,1])$ initial water levels, we obtain
\begin{align*}
F_{\kappa(1)}(x)=\Prob(\kappa(1)\leq x)&=\int_{s=0}^x\limits\int_{t=0}^{\min\{2x-s,1\}}\limits\min\{3x-(s+t),1\}
\,\mathrm{d}t\,\mathrm{d}s\\
                     &=\begin{cases}\tfrac83 x^3\\
                                    -\tfrac{11}{6}x^3+\tfrac92 x^2-\tfrac32 x+\tfrac16\\
                                    -\tfrac{23}{6}x^3+\tfrac{13}{2} x^2-2x+\tfrac16\\
                                    \tfrac{2}{3}x^3-\tfrac52 x^2+4x-\tfrac76\\
                       \end{cases}\text{ for }
                       \begin{cases}x \in [0,\tfrac13]\\x \in [\tfrac13,\tfrac12]\\ x \in [\tfrac12,\tfrac23]
                       \\ x \in [\tfrac23,1]\\ \end{cases}\hspace*{-0.2cm},
\end{align*}
and similarly
\begin{align*}
F_{\kappa(2)}(x)&=\int_{s=0}^x\limits\int_{t=0}^{\min\{2x-s,1\}}\limits\hspace*{-0.5cm}
                  \min\{1,2x-s,3x-(s+t),4x-(2s+t),2x-\tfrac{s+t}{2}\}\,\mathrm{d}t\,\mathrm{d}s\\
                     &=\begin{cases}2x^3\\
                                    -\tfrac{14}{3}x^3+9 x^2-4 x+\tfrac{7}{12}\\
                                    \tfrac23 x^3-3x^2+5x-\tfrac53\\
                       \end{cases}\text{ for }
                       \begin{cases}x \in [0,\tfrac12]\\ x \in [\tfrac12,\tfrac34]
                       \\ x \in [\tfrac34,1]\\ \end{cases}\hspace*{-0.2cm},
\end{align*}                       
which is strictly smaller than $F_{\kappa(1)}(x)$ implying $\kappa(1)\preceq\kappa(2)$, where $\preceq$ denotes
the usual stochastic order.
Due to the fact that adding the edge $\langle1,3\rangle$ will not give an improvement over the optimal move sequences
for vertex $2$, this stochastic domination already follows from the fact that $\kappa(1)$ is non-decreasing when
adding an edge and the symmetry of $K_3$.

In fact, when optimizing the move sequence for the middle vertex we can neglect the option of levelling out
the profile completely, since for any initial water profile there is a {\em finite} optimal move sequence $\phi\in E^T$
achieving $$\eta_T(2)\geq\tfrac13\,\big(\eta_0(1)+\eta_0(2)+\eta_0(3)\big),$$
as the next example will show.
\end{example}

\begin{example}\label{complete}
Given an initial water profile $\{\eta_0(u)\}_{u\in V}$ and the complete graph $K_n$ as underlying network,
we get for any $v\in V$:
$$\kappa(v)=2^{-l+1}\,\eta_0(v)+\sum_{i=1}^{l-1} 2^{-i}\;\eta_0(v_i),$$
where $V$ is ordered such that $\eta_0(v_1)\geq\eta_0(v_2)\geq\dots\geq\eta_0(v_n)$ with $v=v_l$. Furthermore,
this optimal value can be achieved by a finite move sequence.

To see this is not hard having Lemmas \ref{dual} and \ref{SADmax} in mind.
If $v=v_1$, the highest water level is already in $v$ and the best strategy is to stay away from the pipes.
For $v\neq v_1$, the contribution of vertex $v_1$ -- i.e.\ the share $\xi_T(v_1)$ in the convex combination of
$\{\eta_0(u)\}_{u\in V}$ optimizing $\eta_T(v)$, see (\ref{convcomb}) -- can not be more than $\tfrac12$ by Lemma
\ref{SADmax}. However, this can be achieved by opening the pipe
$\langle v,v_1\rangle$.
According to the duality between water transport and SAD, this is what we do last. The argument just used
can be iterated for the remaining share of $\tfrac12$ giving that $v_2$ can contibute at most $\tfrac14$
(given that $v_1$ contributes most possible) and so on.
Obviously, involving vertices holding water levels below $\eta_0(v)$ can not be beneficial, as all vertices are
directly connected, so we do not have intermediate vertices being potential bottlenecks.

The optimal move sequence $\phi\in E^T$, where $T=l-1$, is then given by 
$$\phi_k=\langle v,v_{l-k}\rangle,\ k=1,\dots,l-1$$
leading to 
$$\eta_k(v)=2^{-k}\,\eta_0(v)+\sum_{i=1}^{k} 2^{-k+i-1}\;\eta_0(v_{l-i})$$
and consequently $\eta_T(v)=\eta_{l-1}(v)=\kappa(v)$. Note that the option to open several pipes simultaneously
is useless on the complete graph. Furthermore the above move sequence only includes edges to which $v$ is incident,
so the very same reasoning holds for the center $v$ of a star graph on $n$ vertices as well.

To determine the optimal achievable value at $v$ we have to sort the $n$ initial water levels first. 
This can be done using the randomized sorting algorithm `quicksort' which makes $\text{O}(n\,\log(n)))$
comparisons on average, $\text{O}(n^2)$ in the worst case. The calculation of $\kappa(v)$ given the sorted list of
initial water levels needs at most $n-1$ additions and $n-1$ divisions by $2$.
\end{example}

\begin{example}\label{linegraph2}
Expanding Example \ref{linegraph}, let us reconsider the line graph -- this time not on three
but $n$ vertices. Let the vertices be labelled $1$ through $n$ and let vertex $1$ (sitting at one end of the line) be
the target vertex. Given an initial water profile $\{\eta_0(i)\}_{i=1}^n$, $\kappa(1)$ can be determined by $2n-2$
arithmetic operations ($n-1$ additions, $n-1$ divisions) as it turns out to be
\begin{equation}\label{kappa_line}\kappa(1)=\max_{1\leq l\leq n}\frac1l \sum_{i=1}^l \eta_0(i).\end{equation}
In other words, $\kappa(1)$ is achieved by averaging over the greedy lattice animal for vertex 1 with respect to
the initial water profile (see Definition \ref{GLAdef}).

In order to establish this, let us first define for any water profile $\{\eta_T(i)\}_{i=1}^n$ achievable
from $\{\eta_0(i)\}_{i=1}^n$ (in finite time $T\in\N_0$) the corresponding normed vector
$$\lambda_T:=\tfrac1M\,\Big(\eta_T(1),\dots,\eta_T(n)\Big),$$
which can be understood to be a probability measure on $\{1\dots,n\}$ -- as
(\ref{update}) preserves the total mass $M:=\sum_{i=1}^n\eta_0(i)$.

Related to this construction, there are two important observations to make:
Firstly, if we consider two water profiles on $\{1\dots,n\}$ with the same total mass $M$ and their corresponding
probability measures being such that one stochastically dominates the other, this relation is preserved when the
two profiles are exposed to the same update of the form (\ref{update}), see also La.\ 2.4 in \cite{ShareDrink}.
Secondly, if $\{\eta_{k+1}(i)\}_{i=1}^n$ arises from $\{\eta_{k}(i)\}_{i=1}^n$ by an update of the form (\ref{update})
along an edge $\langle u,u+1\rangle$, where the water level at $u+1$ is higher than the one at $u$, the
corresponding probability measure $\lambda_{k+1}$ is stochastically dominated by $\lambda_k$.

If we choose to open only this kind of pipes (where an update brings water closer to vertex $1$) in a sensible
order -- in permanent sweeps from $\langle 1,2\rangle$ to $\langle n-1, n\rangle$ for example -- we get a
stochastically decreasing sequence of probability measures $(\lambda_k)_{k\in\N_0}$, whose limit is stochastically
dominated by any measure $\tfrac1M\,\big(\eta_T(1),\dots,\eta_T(n)\big)$, where $T\in\N_0$ and
$\{\eta_T(i)\}_{i=1}^n$ was created by updates of the form (\ref{update}) successively applied to the initial
profile $\{\eta_0(i)\}_{i=1}^n$.

Furthermore, following this update scheme the water profiles are tending to a piecewise constant profile:
Any relation ``the barrel at vertex $x$ holds at most as much water as the one at $x+1$''
will be preserved if we only open pipes $\langle u,u+1\rangle$, where $u+1$ has a water level higher than
the one at $u$. For that reason, already the initial water profile determines two sorts of pipes:
If $L$ is minimal in respect of
$$\frac1L \sum_{i=1}^L \eta_0(i)=\max_{1\leq l\leq n}\frac1l \sum_{i=1}^l \eta_0(i)$$
and $L>u$, the water level at $u+1$ will eventually be higher than the one at $u$ causing infinitely many updates
along $\langle u,u+1\rangle$ (or the very same water level at $u$ and $u+1$) in the sequel. If instead $L\leq u$,
either the water level at $u$ is always at least as high as at $u+1$ and $\langle u,u+1\rangle$ will never be opened
or the barrel at $u$ is at some point emptier than the one at $u+1$ leading to infinitely many updates along
$\langle u,u+1\rangle$ or eventually the same water level at $u$ and $u+1$.
This establishes the claimed shape of the limiting profile (according to Lemma \ref{evenout}). Its optimality can be
seen as follows:
The probability measure $\lambda_T$ corresponding to an arbitrary achievable profile $\{\eta_T(i)\}_{i=1}^n$
dominates $\lambda_k$ for $k$ large enough, hence
$$\eta_T(1)= M\cdot\lambda_T(1)\leq M\cdot\lambda_k(1)=\eta_k(1)\leq\lim_{k\to\infty}\eta_k(1),$$
and (\ref{kappa_line}) follows.

If we allowed macro moves (opening several pipes simultaneously), the first (and only) move would be to open
the pipes $\langle1,2\rangle,\dots,\langle L-1,L\rangle$.
\end{example}

\begin{example}\label{linegraph3}
Finally, let us consider the line graph on $n$ vertices, with the target vertex $v$ not sitting at one end.
\begin{figure}[H]
	\centering
	\includegraphics[scale=0.8]{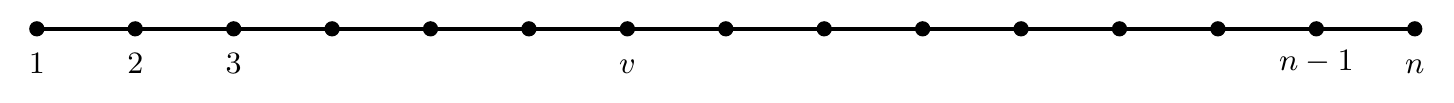}
\end{figure}

Given the initial water levels $\{\eta_0(u),\;1\leq u\leq n\}$, let us consider the final SAD-profile
$\{\xi(u)\}_{1\leq u\leq n}$ corresponding to an optimal move sequence (if the move sequence is of infinite type,
it is the limit of its dual SAD-profiles we are talking about, cf. Lemmas \ref{dual} and \ref{simplif}).

First of all, from Lemma \ref{collection}\,(a) we know that any achievable SAD-profile on line graphs
is unimodal (which therefore holds for a pointwise limit of SAD-profiles as well). Let us denote the leftmost
maximizer of $\{\xi(u)\}_{1\leq u\leq n}$ by $q$ and set
$$l:=\min\{1\leq u\leq n,\;\xi(u)>0\}\quad\text{and} \quad r:=\max\{1\leq u\leq n,\;\xi(u)>0\}.$$
By symmetry, we can assume without loss of generality $l\leq v\leq q\leq r$ -- if $q<v$, the set-up is merely mirrored.
Furthermore, let us pick the optimal move sequence such that $\{\xi(u)\}_{1\leq u\leq n}$ minimizes the distance $q-v$.

The contribution from the nodes $\{q,q+1,\dots,n\}$ can be seen as a scaled-down version of the problem treated
in the previous example: This time the drink to be shared does not amount to 1 but to $\sum_{q\leq u\leq r}\xi(u)$
instead. From Example \ref{linegraph2} we can therefore conclude that a flat SAD-profile i.e.\ 
\begin{equation}\label{flatend}
\xi(q)=\xi(q+1)=\ldots=\xi(r)
\end{equation}
is optimal. The same holds for the contribution coming from $\{1,2,\dots,v-1\}$, i.e.\ 
\begin{equation}\label{flatend2}
\xi(l)=\xi(l+1)=\ldots=\xi(v-1).
\end{equation}
In addition to that, from Lemma \ref{collection}\,(c) we know $\xi(r)\leq\tfrac{1}{r-v+1}$.

If $l=v$, part (b) of Lemma \ref{collection} in turn implies $v=q$. The SAD-profile then
features only one non-zero value (namely $\tfrac{1}{r-v+1}$) and corresponds to the
greedy lattice animal for $v$ consisting of the vertices $v,v+1,\dots,r$. If instead $l<v$ -- compared to
the balanced average across $\{v,v+1,\dots,r\}$ just described -- the contribution to the final water
level at $v$ (cf.\ (\ref{convcomb})) given by
\begin{equation}\label{replacement}
\sum_{u=l}^{v-1}\xi(u)\,\eta_0(u)\quad\text{replaced the contribution}\quad
\sum_{u=v}^{r}\big(\tfrac{1}{r-v+1}-\xi(u)\big)\,\eta_0(u),
\end{equation}
where necessarily $\sum_{u=l}^{v-1}\xi(u)=\sum_{u=v}^{r}\big(\tfrac{1}{r-v+1}-\xi(u)\big)=:M$.
As $q$ is a mode and due to (\ref{flatend}) we have
\begin{equation}\label{decrease}
\tfrac{1}{r-v+1}-\xi(v)\geq\ldots\geq\tfrac{1}{r-v+1}-\xi(q)=\ldots=\tfrac{1}{r-v+1}-\xi(r).
\end{equation}
The aforementioned replacement is most beneficial if the weighted average to the right in (\ref{replacement})
is made as {\em small} as possible, keeping $M$ fixed. In view of (\ref{decrease}) we can conclude, applying
again the ideas from the foregoing example -- this time think of the initial profile $C-\eta_0(u)$
considered for $v\leq u\leq r$ only -- that this is achieved once more by a balanced average. Hence $l<v$ implies
$v<q$ and the just mentioned balanced average has to stretch to the right as far as $q-1$, i.e.\
$\xi(v)=\ldots=\xi(q-1)<\tfrac{1}{r-v+1}=\xi(q)$, since otherwise $q$ would not be the leftmost mode.
From this and (\ref{flatend2}) we find $M=(v-l)\cdot\xi(l)=(q-v)\cdot \big(\tfrac{1}{r-v+1}-\xi(v)\big)$.

The assumption that $q-v$ was minimized when picking the optimal move sequence considered, forces
\begin{equation*}
\sum_{u=l}^{v-1}\xi(u)\,\eta_0(u)\ >\ \sum_{u=v}^{q-1}\big(\tfrac{1}{r-v+1}-\xi(u)\big)\,\eta_0(u),
\end{equation*}
since otherwise the balanced average across $\{v,v+1,\dots,r\}$ would have been at least as good.
Connecting $v$ to barrels to the left consequently yields an improvement of the final water level at $v$
(in comparison to $\tfrac{1}{r-v+1}\sum_{u=v}^{r}\eta_0(u)$) to the amount of
\begin{equation*}
\sum_{u=l}^{v-1}\xi(u)\,\eta_0(u)-\sum_{u=v}^{q-1}\big(\tfrac{1}{r-v+1}-\xi(u)\big)\,\eta_0(u)
=M\cdot\Big(\tfrac{1}{v-l}\sum_{u=l}^{v-1}\eta_0(u)-\tfrac{1}{q-v}\sum_{u=v}^{q-1}\eta_0(u)\Big).
\end{equation*}
As a consequence, $M$ must be as large as possible for an optimal move sequence, which means $\xi(l)=\xi(q-1)$ and
makes $\{\xi(u)\}_{1\leq u\leq n}$ a piecewise constant profile taking on two non-zero values, $\xi(l)$ and 
$\xi(r)$, as depicted below.
\begin{figure}[H]
	\centering
	\includegraphics[scale=0.8]{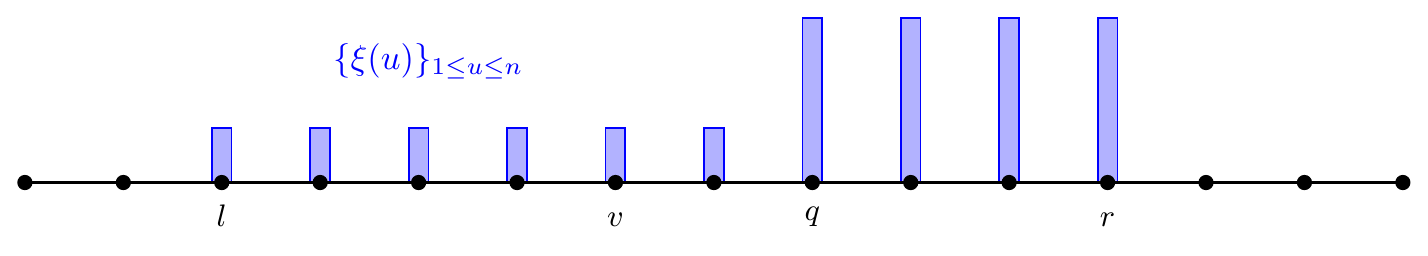}
\end{figure}
Note that the value $\xi(r)=\tfrac{1}{r-v+1}$ (and so even $\xi(l)$) is already determined by the choice
of $l,q$ and $r$. In Figure \ref{15-line2} below, a set of initial water levels on the line graph comprising 15 nodes
is shown, for which the SAD-profile corresponding to an optimal move sequence is the one shown above. Furthermore,
it can be seen from this instance that the GLA with respect to the initial water profile and its possible
enhancements can be outperformed by improving another lattice animal as mentioned at the end of Subsection \ref{optimization}.
\begin{figure}[H]
	\centering
	\includegraphics[scale=0.4]{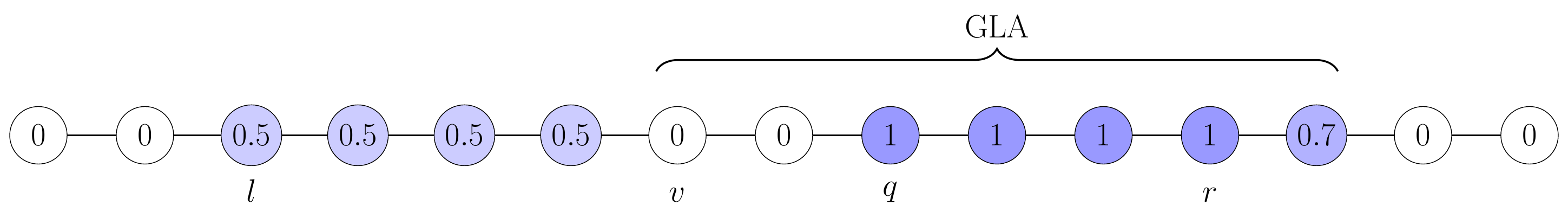}
	\caption{Even for a graph as simple as the line graph, the initial GLA sometimes has little to do with the
		     optimal move sequence.\label{15-line2}}
\end{figure}
When it comes to the complexity of finding $\kappa(v)$, we can greedily test all choices for $l,q,r$ -- of which
there are less than $n^3$. For each choice at most $n+3$ additions/subtractions and four multiplications/divisions have
to be made to calculate either
\begin{equation}\label{value}\begin{array}{c}
\tfrac{q-v}{(q-l)\,(r-v+1)}\sum_{u=l}^{q-1}\limits\eta_0(u)+\tfrac{1}{r-v+1}\sum_{u=q}^{r}\limits\eta_0(u) \quad\text{or}\\
\tfrac{1}{v-l+1}\sum_{u=l}^{\hat{q}}\limits\eta_0(u)+\tfrac{v-\hat{q}}{(r-\hat{q})\,(v-l+1)}\sum_{u=\hat{q}+1}^{r}\limits\eta_0(u),
\end{array}
\end{equation}
depending on whether $v\leq q$ or $\hat{q}\leq v$, where $\hat{q}$ is the rightmost mode of $\{\xi(u)\}_{1\leq u\leq n}$.
Even if there might exist SAD-profiles with $q<v<\hat{q}$ corresponding to optimal move sequences, by the above we know that
there has to be one with either $v\leq q$ or $\hat{q}\leq v$ as well.
The maximal value among those calculated in (\ref{value}) equals $\kappa(v)$, so the complexity is $\mathcal{O}(n^4)$.
\end{example}

\begin{example}\label{seqofseq}
	
The preceding example can serve to give a concrete instance in which even an infinite sequence of single edge moves
can not achieve the supremum as mentioned after Definition \ref{opt}.
\par\begingroup \rightskip11em\noindent Consider the line graph on four vertices, the target
vertex not to be one of the end vertices and initial water levels as depicted to the right.
\par\endgroup

\vspace*{-1.6cm}
\begin{figure}[H]
	\flushright \includegraphics[scale=0.45]{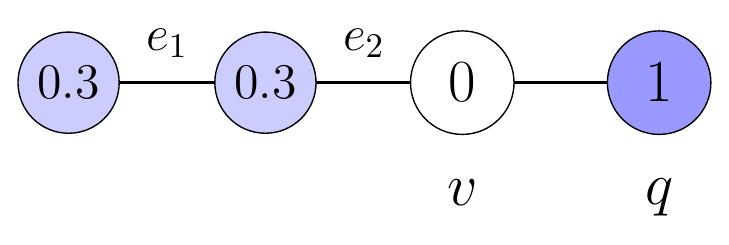}
\end{figure}
\vspace*{-0.45cm}

From Example \ref{linegraph3} we know that the optimal SAD-profile will allocate $\tfrac16$ of the shared glass of water to
each of the vertices to the left of $v$ and $v$ itself, the maximal amount of $\tfrac12$ to the rightmost vertex $q$ -- showing
that $\kappa(v)=0.6$: First, recall that any SAD-profile on a line graph is unimodal. If $q$ is not the (only)
mode, the contribution of $v$ and $q$ has an average of at most $0.5$ and thus the SAD-profile in question
yields a water level at $v$ of at most $0.5$ -- see (\ref{convcomb}). If $q$ is the mode, the SAD-profile is non-decreasing from
left to right and thus a flat profile on the vertices other than $q$ uniquely optimal. Finally, to achieve the optimum, the
contribution of $q$ has to be maximal, i.e.\ $\tfrac12$ (see Lemma \ref{SADmax}).

From the considerations in Thm.\ 2.3 in \cite{ShareDrink} it is clear that this SAD-profile, more precisely the value $\tfrac12$
at $q$, can only be established if the first move is $v$ sharing the drink with $q$ (which corresponds to the last move in the
water transport -- see Lemma \ref{dual}). Once $v$ starts to share the drink to the left, any other interaction with $q$ will
decrease the contribution of the latter and thus put a water level of $0.6$ at $v$ out of reach. 

To get a flat profile on three vertices, we need however infinitely many single-edge moves (here on $e_1$ and $e_2$). An infinite
type optimal move sequence is for example given by
\begin{eqnarray*}
\Phi&=&\{\phi^{(m)},\;m\in\N\},\quad \text{where}\quad \phi^{(m)}\in E^{T_m},\ T_m=2m+1\quad\text{and}\\
\phi^{(m)}&=&(\underbrace{e_1, e_2, e_1, e_2,\dots,}_{m \text{ pairs}}\langle v,q\rangle),
\end{eqnarray*}
achieving
$$\lim_{m\to\infty}\eta_{T_m}(v)=0.6=\kappa(v),$$
a value that can not be approached by any stand-alone (finite or) infinite sequence of moves.

If we allow macro moves, however, there is a two-step move sequence achieving the water level $0.6$ at $v$: First we open the
pipes $e_1$ and $e_2$ simultaneously to pool the water of the vertices other than $q$ and in the second round, we open
the pipe $\langle v,q\rangle$.
\end{example}

\section{Complexity of the problem}\label{complex}

In this section, we want to build on the complexity considerations for the water transport on finite graphs
from the examples in Section \ref{finite}. In fact, we want to show that the task of determining whether $\kappa(v)$
is larger or smaller than a given constant -- for a generic set-up, consisting of a graph, target vertex and initial water
profile -- is an NP-hard problem. This is done by establishing the following theorem:

\begin{theorem}\label{NPhard}
	The NP-complete problem 3-SAT can be polynomially reduced to the decision problem of whether
	$\kappa(v)> c$ or not, for an appropriately chosen water transport instance and constant $c$.
\end{theorem}

Before we deal with the design of an appropriate water transport instance in order to embed the satisfiability
problem {\em 3-SAT}, let us provide the definition of Boolean satisfiability problems as well
as known facts about their complexity.

\begin{definition}
	Let $X=\{x_1,x_2,\dots, x_k\}$ denote a set of {\em Boolean variables}, i.e.\ taking on logic truth values `TRUE' (T)
	and `FALSE' (F). If $x$ is a variable in $X$, $x$ and $\overline{x}$ are called {\em literals} over $X$.
	A {\em truth assignment} for $X$ is a function $t:X\to\{T, F\}$, where $t(x)=T$ means that the variable $x$ is set
	to `TRUE' and $t(x)=F$ means that $x$ is set to `FALSE'. The literal $x$ is true under $t$ if and only if $t(x)=T$,
	$\overline{x}$ is true under $t$ if and only if $t(x)=F$.
	
	A {\em clause} $C$ over $X$ is a disjunction of literals and {\em satisfied} by $t$ if at least one of its literals
	is true under $t$.
	A logic formula $F$ is in {\em conjunctive normal form (CNF)} if it is the conjunction of (finitely
	many) clauses. It is called {\em satisfiable} if there exists a truth assignment $t$ such that all its clauses are
	satisfied under $t$.
	
	The standard Boolean satisfiability problem (often denoted by {\em SAT}) is to decide whether a given formula in CNF is
	satisfiable or not. If we restrict to the case where all the clauses in the formula consist of at most 3 literals it is
	called {\em 3-SAT} 
	.
\end{definition}

3-SAT was among the first computational problems shown to be NP-com\-plete, a result published in a pioneering article by
Cook in 1971, see Thm.\ 2 in \cite{3SAT}. 


	
	

	Let us now turn to the task of embedding 3-SAT into an appropriately designed water transport problem that
	is in size polynomial in $n$, the number of clauses of the given 3-SAT problem:
	
	Given the logic formula $F=C_1\land C_2\land\ldots\land C_n$ in which each of the clauses $C_i$ consists of at most 3
	distinct literals, let us define the comb-like graph depicted in Figure \ref{reduction}. All the white nodes, plus
	the target vertex $v$, represent empty barrels. The other ones that are shaded in blue contain water to the amount
	specified.
	
	\begin{figure}[ht]
		\centering\hspace*{-0.1cm}
		\includegraphics[scale=0.36]{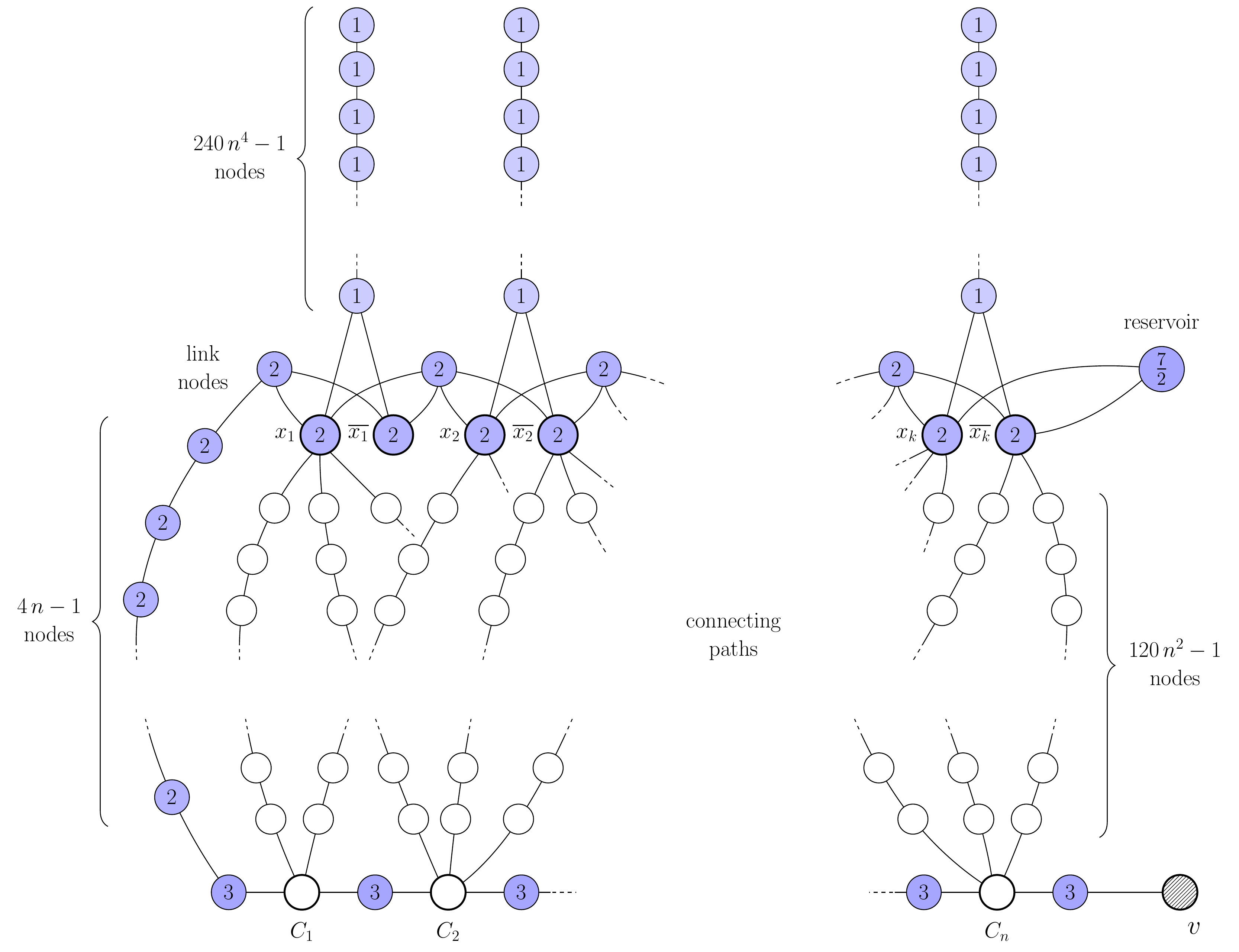}
		\caption{A polynomial reduction of 3-SAT to the water transport problem. \label{reduction}}
	\end{figure}
	
	The comb has $k$ teeth, where $k$ is the number of variables appearing in $F$. Each individual tooth is formed by a line graph
	on $240\,n^4-1$ vertices with water level $1$ each. The lower endvertex of the $i$th tooth is connected to two vertices
	representing the literals $x_i$ and $\overline{x_i}$, having water level $2$ respectively. In between the teeth there are
	$k-1$ {\em link nodes}, each of which features itself a water level of $2$ and is connected to the four nodes representing literals
	of consecutive variables -- more precisely, the link node in between tooth $i$ and $i+1$ is connected to the vertices $x_i,\overline{x_i},x_{i+1},\overline{x_{i+1}}$, for $i\in\{1,\dots,k-1\}$. The vertices representing $x_k,\overline{x_k}$
	are connected to the rightmost link node as well as to an additional vertex featuring a water reservoir of level $\tfrac72$.
	Left of the first tooth, there is another link node (with water level $2$ as well) connected to $x_1$ and $\overline{x_1}$
	as well as by a path to the shaft of the comb, which is described next.	
	
	The comb's shaft is made up of a line graph on $2n+2$ vertices, with the target vertex $v$ to the
	very right. To the left of $v$ there is a vertex representing a barrel with water level $3$ followed by $n$ (empty)
	barrels that stand for the clauses $C_1,\dots,C_n$ and are seperated by a vertex with water level $3$ respectively.
	The left endvertex (connected to $C_1$) features a water level of $3$ as well and is connected to the
	leftmost link node as mentioned before, namely via a path consisting of $4\,n-1$ nodes with water level $2$ each.
	
	Finally, the teeth are connected to the shaft through (disjoint) {\em connecting paths} from nodes
	representing literals to nodes representing clauses, where for example $\overline{x_2}$ is linked to $C_2$ by a path
	if it appears in this clause.
	Each of these paths is formed by a line graph on $120\,n^2-1$ vertices representing empty barrels. Note that each
	clause-node is linked to at most 3 connecting paths, whereas the number of connecting paths originating from a vertex
	representing a literal can vary between $0$ and $n$.\vspace*{1em}
	
	In connection with the water transport problem originating from a 3-SAT formula $F$ as depicted in Figure \ref{reduction},
	we claim the following:
	
	\begin{proposition}\label{decision}
    Consider the water transport problem based on the logical formula $F$, given by the graph, target vertex and initial water
    profile as depicted in Figure \ref{reduction}.    
    \begin{enumerate}[(a)]
    	\item If $F$ is satisfiable, then the water level at $v$ can be raised to a value strictly larger than $2$, i.e.\ 
    	$\kappa(v)>2$.
    	\item If $F$ is not satisfiable, then this is impossible, i.e.\ $\kappa(v)\leq2$.
    \end{enumerate}	
	\end{proposition}	

	Before we deal with the proof of the proposition, note how it implies the statement of Theorem \ref{NPhard}:
	First of all, if $F$ is a 3-SAT formula consisting of $n$ clauses, $k$ cannot exceed $3\,n$. Given this, it is not
	hard to check that the graph in Figure \ref{reduction} has no more than $720\,n^5+360\,n^3+9\,n+2$ vertices and
	maximal degree at most $n+3$ (or $5$ if $n=1$). As the initial water levels are all in $\{0,1,2,3,\tfrac72\}$,
	the size of this water transport instance is clearly polynomial in $n$.
	Due to the fact that the value of $\kappa(v)$ can be used to decide whether the given formula $F$ is satisfiable
	or not -- as claimed by Proposition \ref{decision} -- Theorem \ref{NPhard} follows.
	\vspace*{1em}
	
	\begin{nproof}{of Proposition \ref{decision}$\,\mathrm{(a)}$}
	To prove the first part of the proposition, let us assume that $F$ is satisfiable. Then there exists a truth assignment
	$t$ with the property that all clauses $C_1,\dots,C_n$ contain at least one of the $k$ literals that are set true by $t$.
	Those can be used to let the water trickle down from the teeth to the line graph at the bottom in an effective way:	
	We assign each clause to one of the true literals under $t$ which it contains. Then, we average the water over
	$k$ (disjoint) star-shaped trees. Each such tree has a literal $x\in\{x_1,\overline{x_1},\dots,x_k,\overline{x_k}\}$
	that is true under $t$ as its center and the top node of the tooth above $x$ as well as the nodes representing the clauses
	assigned to $x$ as leaves (where the clause-nodes are connected to $x$ in the tree via the corresponding connecting paths).
	If $m$ clauses chose $x$, there are $240\,n^4+m\cdot 120\,n^2$ vertices in the tree and the water accumulated amounts
	to $240\,n^4+1$.
	
    By pooling the water along those trees, all the nodes corresponding to clauses can simultaneously be pushed to a water
    level as close to the average of the corresponding trees as we like (see Lemma \ref{evenout}). As $m\leq n$, we can
    bound these averages from below by
    $$\frac{240\,n^4+1}{240\,n^4+120\,mn^2}\geq\frac{240\,n^4+1}{240\,n^4+120\,n^3}> 1-\tfrac{1}{2n}.$$
    So after this procedure, each clause-node will have a water level strictly larger than $1-\tfrac{1}{2n}$. Note that
    only one of each pair $\{x_i,\overline{x_i}\}$ was used as a water passage, so there is still a line graph -- let us
    call it {\em linking path} -- consisting of vertices with water level $2$ exclusively, from the leftmost link node
    to the vertex with initial water level $\tfrac72$ through all link nodes and the untouched literals (the ones that are
    false under $t$).
	
	By another complete averaging -- this time over the line graph that consists of the shaft (i.e.\ the line graph at the bottom
	in Figure \ref{reduction}), the path to the very left connecting the leftmost link node to the shaft, the linking path
	just described, as well as the reservoir with level $\tfrac72$ at the other end -- will push the water level at $v$ beyond
	$$\tfrac{1}{6n+2k+2}\,\Big(\tfrac72+(2\,k+4\,n-1)\cdot 2+(n+1)\cdot3+n\,(1-\tfrac{1}{2n})\Big)
	=\frac{12\,n+4\,k+4}{6\,n+2\,k+2}=2.$$
	Consequently, for the case of satisfiable $F$ we verified for the graph depicted in Figure \ref{reduction}:
	$\kappa(v)>2$.
    \end{nproof}

    In the proof of the second part of the proposition, we need a rough estimate of how much the water level in a
    vertex representing a clause can be raised, if only accessed via connecting paths. This is done in the following lemma.
    
    \begin{lemma}\label{treecomp}
    	In the comb-like graph depicted in Figure \ref{reduction}, it is impossible to push the water level in a clause-vertex
    	above the value of $1+\tfrac{1}{2n}$ without opening the pipes to its left or right neighbor.
    \end{lemma}
    
    \begin{proof}
    	The proof of this claim is a simple comparison with a tree similar to the structure above the node corresponding to
    	some clause $C_l$. Originating from $C_l$, there are at most 3 connecting paths that lead to three nodes representing
    	literals. Initially, the node corresponding to $C_l$ and the ones on the connecting paths are empty. Their water level
    	can be raised to almost $1$ using water from the teeth of the comb and further using nodes with initial water level
    	higher than $1$. The fact that opening pipes always produces convex combinations of the involved water levels (see
    	(\ref{convcomb})) guarantees that the total amount of water above a fixed level -- cumulated over all barrels -- is
    	non-increasing when pipes are opened. Initially, the cumulated amount of water above level $1$ in the whole graph is
    	\begin{equation}\label{totalabove1}
    		(\tfrac72-1)+3\,k\cdot(2-1)+(4\,n-1)\cdot(2-1)+(n+1)\cdot(3-1)\leq 15\,n+\tfrac72.
    	\end{equation}
        For $n\in\N$, this is clearly less than $20n-1$.
    	
    	We can mimick any attempt raising the water level at $C_l$ in the comb-graph via its connecting paths
    	in the tree depicted to the right in Figure \ref{pic_treecomp} in such a way, that the water levels at $C_l$ and on
    	its connecting paths are at	any point in time at most as high as the ones in the corresponding part of the comparison tree: 		\begin{figure}[ht]
    		\centering\hspace*{-0.25cm}
    		\includegraphics[scale=0.55]{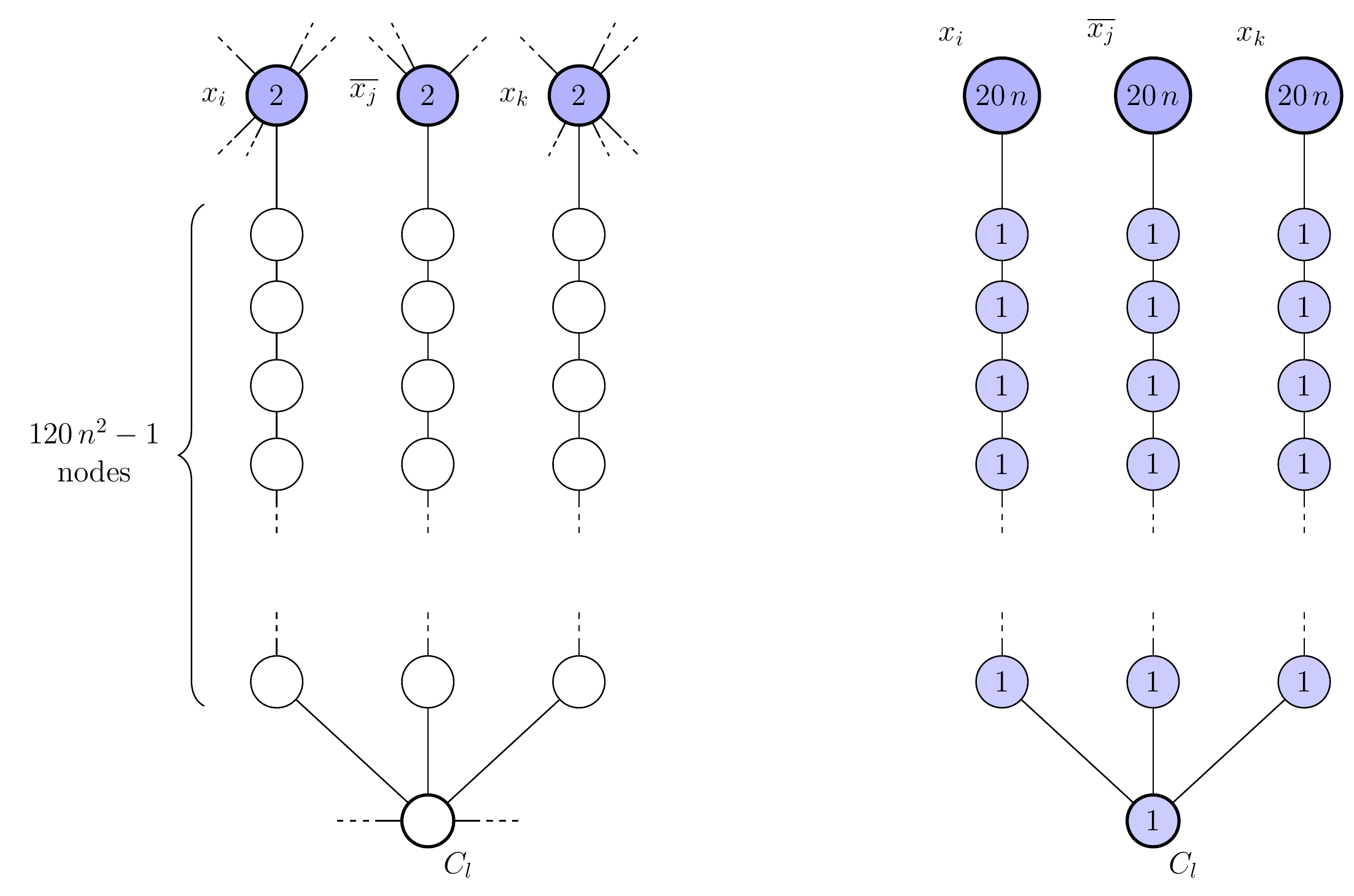}
    		\caption{Comparison of the structure above the node representing a clause $C_l$ in the comb-graph with an
    			appropriately tailored tree. \label{pic_treecomp}}
    	\end{figure}
    	If water is routed into the connecting paths above $C_l$ but water levels do not exceed $1$ (e.g.\ when routing water
    	down from the teeth) we do nothing in the comparison tree. If water from the vertices with initial water level above
    	$1$ is introduced into the connecting paths, we introduce the same amount to the corresponding connecting paths in the
    	tree (note that this is possible, as the total amount of water above level $1$ in the comb-graph is available in all three
    	leaves of the tree). Every move involving only nodes from the connecting paths depicted and $C_l$ is copied in the tree.
    	This retains the property that the water levels in the tree are not less than the ones in corresponding nodes of the
    	comb-graph and shows that the highest water level achievable at $C_l$ in the tree is an upper bound on the level achievable
    	in the comb-graph. If there are less than 3 connecting paths above $C_l$ in the comb-graph we can either modify the
    	comparison tree accordingly or just not use the extra branches.
    	
    	By the generalization of Thm.\ 2.3 in \cite{ShareDrink} to trees, see the comment after Lemma \ref{collection}, we know
    	that the contribution to the convex combination at $C_l$ from the leaves in the tree is at most $1$ divided by the graph
    	distance plus one, i.e.\ $\tfrac{1}{120\,n^2+1}$. The water level at $C_l$ in the tree can therefore not exceed
    	$$1+3\cdot\frac{20\,n}{120\,n^2+1}\leq 1+\frac{1}{2n},$$
    	which induces the claim.
    \end{proof}\vspace*{1em}
    
    Note that the same argument with only water to the amount of $5\,n-1$ above level $1$ available in the leaves would give the
    upper bound of $1+\tfrac{15\,n}{120\,n^2}=1+\tfrac{1}{8n}$, which will be used in the proof of Proposition \ref{decision}\,(b)
    as well.\vspace*{1em}
	
	\begin{nproof}{of Proposition \ref{decision}$\,\mathrm{(b)}$}
	To check that in case $F$ is not satisfiable we get $\kappa(v)\leq2$ is a bit more involved than the first part:
	Let us assume the contrary. Then there exists a finite move sequence (involving macro moves say) that achieves
	a final water level $\eta_T(v)>2$. By the idea in part (a) of Lemma \ref{simplif2} we can assume the last move
	to be the complete average over a connected vertex set $A$ including $v$. The only barrels with initial water
	level larger than $2$ are the ones left of each clause-node and of $v$ plus the reservoir. Including any
	node apart from these into the set $A$, when trying to achieve $\eta_T(v)>2$, can therefore only be beneficial if
	it is a bottleneck (see the discussion after Definition \ref{GLAdef}).
	
	Structurally speaking, there are three potential candidates for such a set $A$:
	\begin{itemize}
		\item a set containing some vertex from a connecting path \vspace*{-0.5em}
		\item a set containing only vertices from the bottom line graph or \vspace*{-0.5em}
		\item a set containing the reservoir vertex but no connecting path.
	\end{itemize}
	Note that the set we used in the case of satisfiable $F$ was of the third type. We will see in a moment that
	this is in fact the only relevant candidate for the set $A$ in the sense that the other two do not allow to raise
	the water level at $v$ above the value of $2$, even for a satisfiable formula $F$.
		
	The first candidate listed is ruled out rather easily: If $A$ contains a vertex from a connecting path, the bottleneck
	argument forces $A$ to contain the whole corresponding connecting path (recall: a bottleneck has to be a cut
	vertex between barrels with water levels above average and the target vertex). Then $A$ is of size at least
	$120\,n^2$ and the amount of water above the water level of $1$ is just not sufficient to fill up so many vertices
	to a level of two: From (\ref{totalabove1}) we know that the water available above level $1$ in the whole graph is
	at most $15\,n+\tfrac72$ initially and non-increasing. The amount in a whole connecting path with water level $2$
	would be $120\,n^2-1$, so this can definitively not be achieved.
	\vspace*{1em}
	
	Next, let us assume that $A$ is a subset of the vertices of the bottom line graph -- including vertex $v$ and $m$ clauses.
	Again by the bottleneck argument, we can assume that the leftmost node in $A$ is not a clause-node, i.e.\ has initial
	water level $3$ (see Figure \ref{setA}).
	\begin{figure}[H]
			\centering
			\includegraphics[scale=0.45]{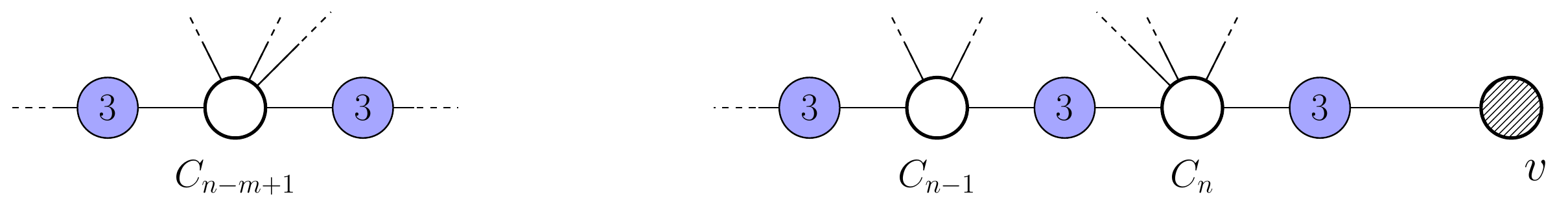}
			\caption{Vertices of the set $A$ considered in the second case. \label{setA}}
		\end{figure}
	With the intention to increase the total amount of water inside $A$ before the final averaging, one can try to fill
	up the clause-nodes. However, from Lemma \ref{treecomp} we know that the water level at the clause-nodes (being bottlenecks)
	in $A$ cannot be pushed much above the level of $1$, if accessed via connecting paths only. Further, this makes accessing
	vertex $v$ through a connecting path and $C_n$ unfavorable. Note that opening the pipes in the bottom line graph
	in order to connect barrels representing clauses inside $A$ to the ones with water level $2$ or $3$ outside $A$ might
	increase the amount of water in $A$ as well, but will raise the water level at the involved clause-nodes to a level
	that can not be further improved by using links	via connecting paths, so it is most beneficial to fill up the
	clause-nodes with water routed through connecting paths first.
	
	Let us assume that after this first phase, we managed to achieve a water level of $1+\tfrac{1}{2n}$ at $C_{n-m+1},\dots,C_n$.
	This might be technically impossible, but surely dominates the water levels achievable using the connecting paths
	only and simplifies our further considerations. Staying away from the connecting paths, water can only be routed into
	$A$ via $C_{n-m}$. If we average over all nodes in $A$ once while doing so, the final averaging is meaningless (because
	then the effect of any move between this and the last move will be increased if again all pipes inside $A$ are opened).
	However, since the last move has to involve $v$ we can assume that any move before leaves the pipe on the edge
	incident to $v$ closed -- and thus w.l.o.g.\ the pipe from $C_n$ towards $v$ as well.
	
	This in turn requires that the connected subset of nodes outside $A$ (incident to the leftmost node in $A$) that
	pools its water with a connected subset inside $A$, including $C_{n-m+1}$ but not $v$, has an average of at least
	$2+\tfrac{1}{4n}$, as the amount of water inside $A$ would decrease otherwise. In view of Lemma \ref{treecomp}, the
	only useful move is therefore to open the pipes along the shaft and through the nodes with initial water level $2$
	(which they actually might have lost during the first phase) in order to connect the vertex with water reservoir
	$\tfrac72$ to $A$. No matter which of the nodes representing literals we include, the water levels in the path connecting
	the leftmost link node to the shaft and the shaft itself will be dominated by the ones obtained if we pretend that
	the water above level $2$ from the reservoir can be transferred to the leftmost link node without any losses.
	
	Starting with a water level of $\tfrac72$ in the leftmost link node instead, we might increase the amount of
	water inside $A$ by at most another $\tfrac32\cdot\tfrac{2m}{2m+4n}\leq\tfrac12$ (as the path has to involve at
	least $4\,n$ nodes outside $A$, the subset inside $A$ is of size at most $2\,m\leq2\,n$ and already has an average
	of at least $2$). Along a line graph, the contribution of the water level from an endvertex to the ones formed as
	convex combinations along the line is decreasing with the graph distance (see Lemma \ref{collection}\,(b)).
	
	Despite our greatest efforts, the total amount of water in $A$ will consequently not exceed the value
	$3\,(m+1)+m\,(1+\tfrac{1}{2n})+\tfrac12\leq 4\,(m+1)$. Since $A$ consists of $2m+2$ vertices, leveling out
    across this set will possibly raise the amount of water in the barrel at $v$ to the level of $2$,
	but not beyond.
	\vspace*{1em}
	
	Finally, consider $A$ to contain all of the shaft as well as the vertex with water reservoir $\tfrac72$, but
	no vertex from a connecting path. Then $A$, being connected, has to contain the path that consists of $4\,n-1$ vertices,
	connecting the shaft to the leftmost link node, as well as a linking path through link nodes and vertices representing
	literals as described above. However, this time -- with $F$ being not satisfiable -- it is impossible to fill up
	all the clause-nodes to a level of about $1$, leaving at least one path between the reservoir and the leftmost link
	node unaffected: In order to reach all clause-nodes, we have to use both $x_i$ and $\overline{x_i}$ as water passage
	for at least one $i\in\{1,\dots,k\}$.
	
	In comparison to the case of satisfiable $F$, we will lose an amount of at least
	$1-\tfrac{1}{2n}$ for each clause-node that is not reached before the final averaging -- but likewise an amount
	of at least $1$ in the linking path for each pair $\{x_i,\overline{x_i}\}$ in which both nodes were used as water
	passage -- since their water level of $2$ reduces to something less than $1$ when water from the tooth above is
	routed through the node all the way down to a clause-node. By the same token as in Lemma \ref{treecomp},
	the clause-nodes can be filled up to a level of at most $1+\tfrac{1}{8n}$ through the connecting paths,
	as the water available outside $A$ above level $1$ is $k$ (from the literals not part of the linking path) plus $1$
	from a vertex representing a literal on the linking path if we need to route through such (and $k+1<5\,n-1$).
	Note that moving water from inside $A$ through a connecting	path to a clause will in fact reduce the amount of water
	in $A$. Consequently, the set $A$ (which is the same as the one chosen for satisfiable $F$) still contains $6n+2k+2$
	vertices, but the amount of water we can allocate in $A$ is at most
	\begin{eqnarray*}&&\tfrac72+(2\,k+4\,n-1)\cdot 2+(n+1)\cdot3+n\,(1+\tfrac{1}{8n})-(1-\tfrac{1}{2n})\\
	&=&12\,n+4\,k+\tfrac{29}{8}+\tfrac{1}{2n}\\
	&<& 12\,n+4\,k+4,
	\end{eqnarray*} for $n\geq2$.
	Thus, even in this manner we can not raise the water level at vertex $v$ to a level of $2$ or above if $F$ is
	not satisfiable which contradicts the above assumption and in consequence verifies $\kappa(v)\leq2$ for this case.
  \end{nproof}

  As already mentioned, this shows that solving the decision problem ``$\kappa(v)>2$ or $\kappa(v)\leq2$'' for the
  comb-like graph depicted in Figure \ref{reduction} solves the corresponding 3-SAT problem as well. Since 3-SAT is an
  NP-complete problem, we hereby established that any problem in NP can be polynomially reduced to a decision problem
  minor to the computation of $\kappa(v)$ in a suitable water transport instance -- showing that computing $\kappa(v)$
  in general is indeed an NP-hard problem.

\section{On infinite graphs}\label{infinite}

This last section is devoted to the water transport problem on infinite graphs. We consider an infinite, connected, simple
graph $G=(V,E)$ with bounded maximal degree. The initial water levels $\{\eta_0(u)\}_{u\in V}$ are considered to be i.i.d.\ 
with a (non-degenerate) common marginal distribution concentrated on $[0,C]$, for some $C>0$. The supremum $\kappa(v)$ of
achievable water levels at a fixed target vertex $v\in V$ depends on the initial water levels of course, which makes it a
random variable as well.

When the vertices of an infinite graph are assigned individual values, the most natural definition of a {\em global average}
across the graph is to look at a fixed sequence of nested subsets of the vertex set, with the property that every vertex
is included eventually, and then consider the limit of averages across those subsets (if it exists).

Given i.i.d.\ initial water levels, the strong law of large numbers tells us that the randomness of the global average
-- which is non-degenerate on finite graphs -- becomes degenerate if we consider infinite graphs, where it will a.s.\ equal
the expectation of the marginal distribution. $\kappa(v)$ however shows a slightly different behavior:
In order to determine whether the supremum of achievable water levels at a given vertex $v$ is a.s.\ 
constant or not, we have to investigate the global structure of the infinite graph a bit more closely.

If the graph contains a half-line with sufficiently many extra vertices attached to it, the distribution of $\kappa(v)$ 
becomes degenerate for all $v\in V$ -- as stated in Theorem \ref{inf} and the final remark: One can in fact, with
probability $1$, push the water level at $v$ to the essential supremum of the marginal distribution. The infinite line
graph however is too lean to feature such a substructure and behaves therefore much more like a finite graph, in the sense
that the distribution of $\kappa(v)$ is non-degenerate -- see Theorem \ref{qtgraphs}. In order to evolve these two main
results of this section, let us first properly define what we mean by ``sufficiently many extra vertices''.

\begin{definition}\label{hl}
Let $G=(V,E)$ be an infinite connected simple graph. It is said to contain a {\em neighbor-rich half-line}, if there
exists a subgraph of $G$ consisting of a half-line
$$H=\big(\{v_k,\;k\in\N\},\{\langle v_k, v_{k+1}\rangle,\;k\in\N\}\big)$$
and distinct vertices $\{u_k,\;k\in\N\}$ from $V\setminus\{v_k,\;k\in\N\}$ such that there is an
injective function $f:\N\to\N$ with the following two properties (cf.\ Figure \ref{half_line}):
\begin{enumerate}[(i)]
\item For all $k\in\N$: $\langle u_k, v_{f(k)}\rangle\in E$, i.e.\ the vertices $u_k$ and $v_{f(k)}$ are
neighbors in $G$.
\item The function $f$ is growing slowly enough in the sense that $\sum_{k=1}^\infty\frac{1}{f(k)}$ diverges.
\end{enumerate}
\end{definition}

\begin{figure}[H]
	\centering
	\includegraphics[scale=0.88]{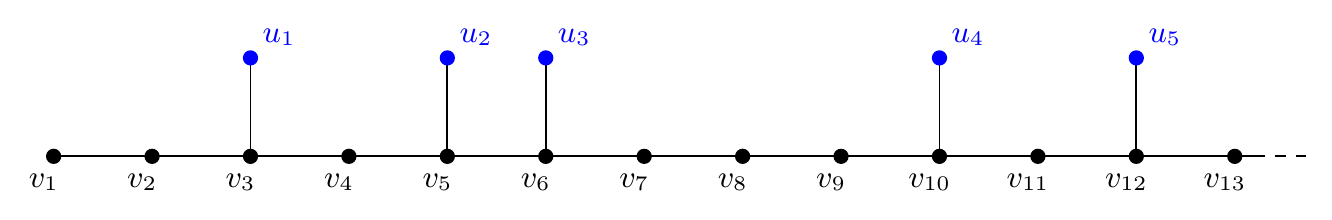}
	\caption{The beginning part of a neighbor-rich half-line.\label{half_line}}
\end{figure}

Note that -- by a renumbering of $\{u_k, \;k\in\N\}$ -- we can always assume the function $f$ to be (strictly)
increasing. Furthermore, if $G$ is connected and contains a neighbor-rich half-line, we can choose any vertex
$v\in V$ to be its beginning vertex: If $v_l$ is the vertex with highest index at shortest distance to $v$ in
$H$, replace $(v_1,\dots,v_l)$ by a shortest path from $v$ to $v_l$ in $H$. The altered half-line will still
be neighbor-rich, since for all $M,N\in\N$ and $f$ as above:
$$\sum_{k=1}^\infty\frac{1}{f(k)}=\infty\quad\Longleftrightarrow\quad\sum_{k=M}^\infty\frac{1}{f(k)+N}=\infty.$$

\noindent
With this notion in hand, we can state and prove the following result:

\begin{theorem}\label{inf}
Consider an infinite (connected) graph $G=(V,E)$ and the initial water levels to be i.i.d.\ $\text{\upshape
unif}([0,1])$. Let $v\in V$ be a fixed vertex of the graph.
If $G$ contains a neighbor-rich half-line, then $\kappa(v)=1$ almost surely.
\end{theorem}

Before embarking on the proof of this theorem, we are going to show a standard auxiliary result which will be
needed in the proof:

\begin{lemma}\label{diverg}
For $\epsilon>0$, let $(Y_k)_{k\in\N}$ be an i.i.d.\ sequence having Bernoulli distribution with parameter
$\epsilon$. If the function $f:\N\to\N$ is strictly increasing and such that $\sum_{k=1}^\infty\frac{1}{f(k)}$
diverges, then
$$\sum_{k=1}^\infty \frac{Y_k}{f(k)}=\infty\quad\text{almost surely.}$$
\end{lemma}

\begin{proof}
Let us define
$$X_n=\sum_{k=1}^n \frac{Y_k-\epsilon}{f(k)}\quad \text{for all }n\in\N.$$
As the increments are independent and centered, this defines a martingale with respect to the natural filtration.
Furthermore,
$$\E(X_n^2)=\sum_{k=1}^n \frac{\E(Y_k-\epsilon)^2}{f(k)^2}=(\epsilon-\epsilon^2)\cdot\sum_{k=1}^n\frac{1}{f(k)^2}
\leq\epsilon\,\frac{\pi^2}{6}.$$
By the $L^p$-convergence theorem (see for instance Thm.\ 5.4.5 in \cite{Durrett}) there exists a random variable $X$
such that $X_n$ converges to $X$ almost surely and in $L^2$. Having finite variance, $X$ must be a.s.\ real-valued
and due to
$$\sum_{k=1}^n \frac{Y_k}{f(k)}-X_n=\epsilon\cdot\sum_{k=1}^n \frac{1}{f(k)},$$
the divergence of $\sum_{k=1}^\infty\frac{1}{f(k)}$ forces
$\sum_{k=1}^\infty \frac{Y_k}{f(k)}=\infty$ almost surely.
\end{proof}\vspace*{1em}

\begin{nproof}{of Theorem \ref{inf}}	
Given a graph $G$ with the properties stated and a vertex $v$, we can choose a neighbor-rich half-line $H$ with
$v=v_1$ and the set of extra neighbors $\{u_n{}\}_{n\in\N}$ as described in and after Definition \ref{hl}.
The initial water levels at $\{u_n\}_{n\in\N}$ are i.i.d.\ $\text{\upshape unif}([0,1])$, of course.

Depending on the random initial profile, let us define the following SAD-procedure starting at $v$:
Fix $\epsilon,\delta>0$ and let $\{N_l\}_{l\in\N}$ be the increasing (random) sequence of indices chosen such that the initial
water level at $u_{N_l}$ is at least $1-\epsilon$ for all $l$. Then define the SAD-procedure --
starting with $\xi_0(v)=1,\ \xi_0(u)=0$ for all $u\in V\setminus\{v\}$ -- such that
first all vertices along the line $(v_1,v_2,\dots,v_{f(N_1)},u_{N_1})$ exchange liquids sufficiently often to get
$$\xi_{k_1}(u_{N_1})\geq\frac{1}{f(N_1)+2}\quad\text{for some }k_1>0,$$
and never touch $u_{N_1}$ again. Note that by Lemma \ref{evenout}, $\xi_{k}(u_{N_1})$ can be pushed
as close to $\tfrac{1}{f(N_1)+1}$ as desired in this way.
At time $k_1$, the joint amount of water in the glasses at $v_1,v_2,\dots,v_{f(N_1)}$ equals $1-\xi_{k_1}(u_{N_1})$ and
we will repeat the same procedure along $(v_1,v_2,\dots,v_{f(N_2)},u_{N_2})$ to get
$$\xi_{k_2}(u_{N_2})\geq\frac{1}{f(N_2)+2}\cdot\big(1-\xi_{k_1}(u_{N_1})\big)\quad\text{for some }k_2>k_1$$
and iterate this.

After $m$ iterations of this kind, the joint amount of water localized at vertices of the half-line $H$ equals
$1-\sum_{l=1}^m \xi_{k_l}(u_{N_l})$, which using $1-x\leq\text{e}^{-x}$ can be bounded from above as follows:
\begin{align}\label{uppbound}\begin{split}
1-\sum_{l=1}^m \xi_{k_l}(u_{N_l})&\leq\prod_{l=1}^m \bigg(1-\frac{1}{f(N_l)+2}\bigg)\\
                               &\leq \exp\bigg(-\sum_{l=1}^m\frac{1}{f(N_l)+2}\bigg).\end{split}
\end{align}
Defining $Y_k:=\mathbbm{1}_{\{\eta_0(u_k)\geq 1-\epsilon\}}$ for all $k\in\N$ we get $(Y_k)_{k\in\N}$ i.i.d.\ 
$\text{Ber}(\epsilon)$ and can rewrite the limit of the sum in the exponent as follows:
$$\sum_{l=1}^\infty\frac{1}{f(N_l)+2}=\sum_{k=1}^\infty\frac{Y_k}{f(k)+2}.$$
This allows us to conclude from Lemma \ref{diverg} that the exponent in (\ref{uppbound}) tends a.s.\ to $-\infty$
as $m \to \infty$. Consequently, $m,T\in \N$ can be chosen large enough such that with probability $1-\delta$ it holds that
$$\sum_{l=1}^m \xi_{k_l}(u_{N_l})\geq 1-\epsilon\quad\text{and}\quad k_m\leq T.$$

Given this event, the move sequence corresponding to the SAD-procedure just described -- adding no further
updates after time $k_m$, i.e.\ $\mu_k=0$ for $k> k_m$, if $k_m<T$ -- then ensures
(see Lemma \ref{dual}) that
$$\eta_T(v)\geq\sum_{l=1}^m\xi_{T}(u_{N_l})\,\eta_0(u_{N_l})\geq(1-\epsilon)^2,$$
forcing $\kappa(v)\geq(1-\epsilon)^2$ with probability at least $1-\delta$. Since $\delta>0$ was arbitrary, this implies
$\kappa(v)\geq(1-\epsilon)^2$ a.s.\ and letting $\epsilon$ go to $0$ then establishes the claim.
\end{nproof}

Let us now take a look at how this result can be used to crystallize the outstanding leanness of the infinite line
among all infinite quasi-transitive graphs. To this end, let us first repeat the definition of quasi-transitivity.

\begin{definition}\label{quasitrans}
Let $G=(V,E)$ be a simple graph. A bijection $f:V\to V$ with the property that $\langle f(u),f(v)\rangle\in E$
if and only if $\langle u,v\rangle\in E$ is called a {\em graph automorphism}. $G$ is said to be {\em (vertex-)
transitive} if for any two vertices $u,v\in V$ there exists a graph automorphism $f$ that maps $u$ on $v$, i.e.\
$f(u)=v$.

If the vertex set $V$ can be partitioned into finitely many classes such that for any two vertices $u,v$
belonging to the same class there exists a graph automorphism that maps $u$ on $v$, the graph $G$ is called
{\em quasi-transitive}.
\end{definition}

Note that the notion of quasi-transitivity becomes meaningful only for infinite graphs as all finite graphs
are quasi-transitive by definition.

\begin{theorem}\label{qtgraphs}
Consider an infinite (connected) quasi-transitive graph $G=(V,E)$ and the initial water levels to be i.i.d.\
$\text{\upshape unif}([0,1])$. Let $v\in V$ be a fixed vertex of the graph.
If $G$ is the line graph, that is $V=\Z$ and $E=\{\langle u, u+1\rangle,\;u\in\Z\}$, then $\kappa(v)$
depends on the initial profile.
If $G$ is not the line graph, then $\kappa(v)=1$ almost surely.
\end{theorem}

\begin{proof}
	Given i.i.d.\ $\text{\upshape unif}([0,1])$ initial water levels, we can immediately conclude two things:
	If $G$ is an infinite (connected) graph, the strong law of large numbers guarantees $\kappa(v)\geq\tfrac12$
	almost surely.
	
	If $G$ is the infinite line graph, there is a positive probability that the vertex $v$ is what Häggström \cite{ShareDrink}
	calls two-sidedly $\epsilon$-flat with respect to the initial profile (see La.\ 4.3 in \cite{ShareDrink}), i.e.
	\begin{equation}\label{tflat}
	\frac{1}{m+n+1}\sum_{u=v-m}^{v+n}\eta_0(u)\in\left[\tfrac12-\epsilon,\tfrac12+\epsilon\right]\quad\text{for all }m,n\in\N_0.
	\end{equation}
	
	La.\ 6.3 in \cite{ShareDrink} states that in this situation, the water level at $v$ is bound to stay within the interval
	$[\tfrac12-6\epsilon,\tfrac12+6\epsilon]$ irrespectively of future updates. Together with the simple observation
	$\kappa(v)\geq\eta_0(v)$, it implies that $\kappa(v)$ is a random variable with non-degenerate distribution
	on $[\tfrac12,1]$.
	
	In view of Theorem \ref{inf}, to prove the second part, we only have to verify, that an infinite, connected,
	quasi-transitive graph that is not the line graph contains a neighbor-rich half-line. Since $G$ is infinite (and by our
	general assumptions both connected and having finite maximal degree) a compactness argument guarantees the existence of
	a half-line $H$ on the vertices $\{v_k,\;k\in\N\}$ such that $v_1=v$ and the graph distance from $v_k$ to $v$ is $k-1$ for
    all $k$.
	
	Let us consider the function $d:V\to\N_0$, where $d(u)$ is the graph distance from the node $u$ to a vertex of degree
	at least 3 being closest to it.
	Since $G$ is quasi-transitive, connected and not the line graph, $d$ is finite and can take on only
	finitely many values, which is why it has to be bounded, by $C\in \N$ say. Consequently, $G$ can not contain stretches of
	more than $2C$ linked vertices of degree 2. For this reason, there must be a vertex among $v_3,\dots,v_{2C+3}$,
	say $v_{f(1)}$, having a neighbor $u_1$ outside of $H$. In the same way, we can find a vertex $u_2$ outside $H$
	having a neighbor $v_{f(2)}$ among $v_{2C+6},\dots,v_{4C+6}$ and in general some $u_k$ not part of
	$H$ but linked to a vertex $v_{f(k)}\in\{v_k,\;k\in\N\}$ with $$(k-1)\,(2C+3)+3\leq f(k)\leq k\,(2C+3)\quad\text{for all }k\in\N.$$
	This choice makes sure that $v_{f(j)}$ and $v_{f(k)}$ are at graph distance at least 3 for $j\neq k$, which forces the
	set $\{u_k,\;k\in\N\}$ to consist of distinct vertices.	Due to
	$$\sum_{k=1}^\infty\frac{1}{f(k)}\geq\frac{1}{2C+3}\,\sum_{k=1}^\infty\frac{1}{k}=\infty,$$
	$H$ is a neighbor-rich half-line in the sense of Definition \ref{hl} as desired.
\end{proof}

\begin{remark}
\begin{enumerate}[(a)]
\item
Note that the essential property of the initial water levels, needed in the proof of Theorem \ref{inf}, was
independence. The argument can immediately be generalized to the situation where the initial water levels are independently (but not
necessarily identically) distributed on $[0,C]$ and we have some weak form of uniformity, namely:

For every $\delta>0$, there exists some $\epsilon>0$ such that
for all $v\in V$: $$\Prob\big(\eta_0(v)>C-\delta\big)\geq \epsilon.$$
The sequence $Y_k:=\mathbbm{1}_{\{\eta_0(u_k)\geq C-\delta\}},\ k\in\N,$ similar to the one defined in the proof of
Theorem \ref{inf} will no longer be i.i.d.\ $\text{Ber}(\epsilon)$, but an appropriate coupling will ensure that
$$\sum_{k=1}^\infty \frac{Y_k}{k}\geq\sum_{k=1}^\infty \frac{Z_k}{k}$$
almost surely, where $(Z_k)_{k\in\N}$ is an i.i.d.\ sequence of $\text{Ber}(\epsilon)$ random variables. Accordingly,
we get $\kappa(v)=C$ a.s.\ even in this generalized setting.
\item
As alluded to in the introduction, the statement of Theorem \ref{qtgraphs} can be interpreted in the following way:
When it comes to the qualitative behavior of $\kappa(v)$ for a fixed vertex $v$ in the graph, the radical change does
not happen between finite and infinite graphs but rather between the line graph $\Z$ and all other quasi-transitive
infinite graphs, which is why the results for the Deffuant model on $\Z$ can not immediately be transferred to
higher-dimensional grids -- as discussed in the introduction of Sect.\ 3 in \cite{Deffuant}.
\item
Finally, it is worth emphasizing that Theorem \ref{qtgraphs} does not capture the full statement of Theorem \ref{inf}:
If we take the infinite line graph $\Z$ and add an extra neighbor to every node that corresponds to a prime number,
the only quasi-transitive subgraph contained is the line graph itself. However, since it contains a neighbor-rich
half-line, Theorem \ref{inf} states that $\kappa(v)=1$ for i.i.d.\ $\text{\upshape unif}([0,1])$ initial water levels
and any target vertex $v$.
\end{enumerate}
\end{remark}



\vspace{0.5cm}
\makebox[\textwidth][c]{
\begin{minipage}[t]{1.2\textwidth}
\begingroup
	\begin{minipage}[t]{0.5\textwidth}
	{\sc \small Olle Häggström\\
   Department of Mathematical Sciences,\\
   Chalmers University of Technology,\\
   412 96 Gothenburg, Sweden.}\\
   olleh@chalmers.se
	\end{minipage}
	\hfill
	\begin{minipage}[t]{0.5\textwidth}
	{\sc \small Timo Hirscher\\
   Department of Mathematical Sciences,\\
   Chalmers University of Technology,\\
   412 96 Gothenburg, Sweden.}\\
   hirscher@chalmers.se\\
	\end{minipage}
	\endgroup
\end{minipage}}

\end{document}